\documentclass{article}

\usepackage{geometry}
\usepackage{amsmath,amssymb,amsthm, comment, enumitem, array}
\usepackage{color, xcolor}
\usepackage{parskip}
\usepackage{hyperref}
\usepackage{parskip}
\usepackage{setspace}
\usepackage{graphicx}
\usepackage{mathtools}
\usepackage[thinc]{esdiff}
\usepackage[euler]{textgreek}
\usepackage{multicol}

\usepackage{subfig}
\usepackage[mathlines]{lineno}
\usepackage{bbold}
\usepackage{comment}
\newtheorem{theorem}{Theorem}[section]
\newtheorem{proposition}[theorem]{Proposition}

\newtheorem{corollary}[theorem]{Corollary}
\newtheorem{lemma}[theorem]{Lemma}

\newtheorem*{definition*}{Definition}

\newcommand{\rspex}{\mathrm{spex}}

\newcommand{\raspex}{\mathrm{spex}_{\alpha}}

\title{A general theorem in spectral extremal graph theory}
\author{John Byrne\thanks{University of Delaware, Department of Mathematical Sciences, \texttt{jpbyrne@udel.edu}} \and Dheer Noal Desai\thanks{University of Wyoming, Department of Mathematics and Statistics.
The University of Memphis, Department of Mathematical Sciences, \texttt{dndesai@memphis.edu}} \and Michael Tait\thanks{Villanova University, Department of Mathematics \& Statistics, \texttt{michael.tait@villanova.edu}. Research partially supported by NSF grant DMS-2245556.}}

\begin{document}

\maketitle

\begin{abstract}
   The extremal graphs $\mathrm{EX}(n,\mathcal F)$ and spectral extremal graphs $\mathrm{SPEX}(n,\mathcal F)$ are the sets of graphs on $n$ vertices with maximum number of edges and maximum spectral radius, respectively, with no subgraph in $\mathcal F$. We prove a general theorem which allows us to characterize the spectral extremal graphs for a wide range of forbidden families $\mathcal F$ and implies several new and existing results. In particular, whenever $\mathrm{EX}(n,\mathcal F)$ contains the complete bipartite graph $K_{k,n-k}$ (or certain similar graphs) then $\mathrm{SPEX}(n,\mathcal F)$ contains the same graph when $n$ is sufficiently large. We prove a similar theorem which relates $\mathrm{SPEX}(n,\mathcal F)$ and $\mathrm{SPEX}_\alpha(n,\mathcal F)$, the set of $\mathcal F$-free graphs which maximize the spectral radius of the matrix $A_\alpha=\alpha D+(1-\alpha)A$, where $A$ is the adjacency matrix and $D$ is the diagonal degree matrix.
\end{abstract}

\section{Introduction} \label{Section Intro}

For a family of graphs $\mathcal{F}$ and $\alpha \in [0,1)$, we define $\mathrm{ex}(n, \mathcal{F})$, $\mathrm{spex}(n,\mathcal{F})$, and $\mathrm{spex}_\alpha(n, \mathcal{F})$ to be the maximum number of edges, maximum spectral radius of the adjacency matrix, and maximum spectral radius of the matrix $A_\alpha$ respectively over all $n$-vertex graphs which do not contain any $F\in \mathcal{F}$ as a subgraph. Here $A_\alpha = \alpha D + (1-\alpha)A$ where $D$ is the diagonal degree matrix and $A$ is the adjacency matrix. So $\mathrm{spex}(n, \mathcal{F}) = \mathrm{spex}_0 (n, \mathcal{F})$ and $\mathrm{spex}_{1/2}(n,\mathcal{F})$ is one half of the maximum spectral radius of the signless Laplacian over $\mathcal{F}$-free graphs. As is standard, we define $\mathrm{EX}(n, \mathcal{F})$, $\mathrm{SPEX}(n, \mathcal{F})$, and $\mathrm{SPEX}_\alpha(n, \mathcal{F})$ as the set of extremal graphs for each respective function, and we call the graphs in each set \textit{(edge) extremal, spectral extremal,} and \textit{alpha spectral extremal}, respectively. When $\mathcal{F} = \{F\}$ is a single graph we use the standard notation $\mathrm{ex}(n, F)$ instead of $\mathrm{ex}(n, \{F\})$ and similarly with the other functions. If the family of extremal graphs consists of a single graph $H$ we will write $\mathrm{EX}(n, \mathcal{F}) = H$ instead of $\mathrm{EX}(n, \mathcal{F}) = \{H\}$, and similarly with the other functions.  Our goal in this paper is to prove general theorems which capture many previous and new results on these parameters.

Spectral extremal graph theory has become very popular recently. Before the field was studied systematically, several sporadic results appeared over the years, for example, by Wilf \cite{Wilf86} and Nosal \cite{nosal1970eigenvalues}. All results in the area which regard the adjacency matrix can be understood under the general framework of {\em Brualdi-Solheid} problems \cite{BS}, where one seeks to maximize the spectral radius  of the adjacency matrix over a family of graphs. When the family is defined by a forbidden subgraph, this problem has a Tur\'an-type flavor, and the spectral Tur\'an problem was introduced systematically by Nikiforov \cite{nikiforov2010spectral}. Since then many papers have been written, and we refer to the excellent recent survey \cite{LLFsurvey}. We also note that there are hypergraph versions of these problems (see for example \cite{KLM, Nikiforovhypergraphs}), but we only consider graphs in this paper.

Researchers have also frequently considered other matrices besides the adjacency matrix. Signless Laplacian versions of Brualdi-Solheid problems were studied extensively and Nikiforov introduced the $A_\alpha$ matrix as a way to merge these two theories \cite{nikiforov2017merging}.

With so much recent activity in this area, we believe that it is valuable to have general theorems that apply to families of forbidden graphs. One question that naturally presents itself when reviewing previous results is: when are the edge-extremal graphs the same as the spectral-extremal graphs? That is, when do $\mathrm{EX}(n, F)$ and $\mathrm{SPEX}_\alpha(n, F)$ either coincide or overlap?

Some of the first results in the area suggest that perhaps the problems are the same. For example, the spectral Tur\'an theorem \cite{Nikiforov07} shows $\mathrm{EX}(n, K_{r+1}) = \mathrm{SPEX}(n, K_{r+1})$, and \cite{nikiforov2008spectral} shows that $\mathrm{EX}(n, C_{2k+1}) = \mathrm{SPEX}(n, C_{2k+1}) = K_{\lfloor n/2\rfloor, \lceil  n/2\rceil}$ for large enough $n$. Furthermore, for any non-bipartite $F$, the theorems in \cite{ES66, ES46, nikiforov2009spectral} show that the asymptotics of the two functions coincide. However, the two problems can also be vastly different, for example when $F$ is an even cycle \cite{cioabua2022evencycle, nikiforov2010spectral, zhai2020spectralhexagon, zhai2012proof}. One might then guess that there is a difference between $F$ bipartite and non-bipartite, but it is a bit more subtle than that. In \cite{cioabua2022spectral}, it was shown that when forbidding large enough odd wheels the two extremal families have similar structure but are in fact disjoint. 

Hence we seek general criteria which can determine whether the extremal families overlap or not. We mention a few very nice papers along these lines. The first result is in \cite{wang2023conjecture}, proving a conjecture in \cite{cioabua2022spectral}, where it was shown that for any graph $F$ such that any graph in $\mathrm{EX}(n, F)$ is a Tur\'an graph plus $O(1)$ edges, then for large enough $n$ one has that $\mathrm{SPEX}(n, F) \subseteq \mathrm{EX}(n, F)$. Second, in \cite{NWK} a general theorem was proved when forbidding disjoint copies of a graph and in \cite{FL} a general theorem was proved when forbidding edge blow-ups of graphs. Finally, in the very recent paper \cite{FLSZ} the authors prove general results about trees.

In this paper we prove some general results which apply to many forbidden graphs or families. As special cases of our theorems, we recover results from \cite{CLLYZ, CLZ, CLZ2, CLZ3, CLZ4, CLZ5, cioabua2022erdossos, csikvari, FLSZ, FYZ, GH, HLF, LY, LYZ, LZ, LK, nikiforov2010spectral, T, tait2017three, TCC, YS, WZ, ZHW} and additionally we can apply the theorems to prove several new results. More details about specific applications are given after statements of our theorems.

When studying a classical Tur\'an-type problem, every edge increases the objective function by the same amount. But when trying to determine $\mathrm{spex}(n, F)$, some edges will contribute more to the spectral radius than others, and so vertices of large degree are incentivized. When $\alpha>0$, the presence of the diagonal degree matrix incentivizes large degree vertices even further. This paper is an attempt to formalize this phenomenon. 

The rest of this paper is organized as follows. In the rest of Section \ref{Section Intro} we define our notation and state our main results. In Section \ref{Section EX SPEX} we prove results relating $\mathrm{EX}(n,\mathcal F)$ and $\mathrm{SPEX}(n,\mathcal F)$. In Section \ref{Section SPEX ALPHA SPEX}, we prove results relating $\mathrm{SPEX}(n,\mathcal F)$ and $\mathrm{SPEX}_\alpha(n,\mathcal F)$. In Section \ref{Section Conclusion} we give concluding remarks.

\subsection{Notation}

For graphs $G$ and $H$, $\overline{G}$ denotes the complement of $G$, $G+H$ denotes the join of $G$ and $H$, $G\cup H$ denotes their disjoint union, and $a\cdot G$ denotes the disjoint union of $a$ copies of $G$; if $a=\infty$ then this means countably many copies. We write $H\subseteq G$ if $H$ is a subgraph of $G$. We write $\Delta(G)$ for the maximum degree of $G$. We write $A=A(G)$ for the adjacency matrix, $D=D(G)$ for the diagonal degree matrix, and $A_\alpha=A_\alpha(G)=\alpha D(G)+(1-\alpha)A(G)$. For a symmetric matrix $M$ we use $\lambda_1(M)$ to mean the largest eigenvalue of the matrix, and for a graph $G$ we write $\lambda(G)$ and $\lambda_\alpha(G)$ for $\lambda_1(A(G))$ and $\lambda_1(A_\alpha(G))$, respectively. For $v\in V(G)$ and $U\subseteq V(G)$, we write $N(v)$ for the neighborhood of $v$, $N_U(v)$ for $N(v)\cap U$, and $d_U(v)$ for $|N_U(v)|$. For $i\in\mathbb N$, we write $N_i(v)$ for the set of vertices at distance $i$ from $v$. For $U,W\subseteq V(G)$, we write $e(U,W)$ for $|\{(u,w):u\in U,w\in W,u\sim w\}|$, and $e(U)$ for $e(U,U)/2$; we write $E(U,W)$ and $E(U)$ for the corresponding sets of edges. We write $G[U]$ for the subgraph induced by $U$ and if $U\cap W=\emptyset$ we write $G[U,W]$ for the bipartite subgraph on vertices $U\sqcup W$ and with all edges in $G$ between $U$ and $W$. Our notation for specific classes of graphs is standard; we note that $P_n$ denotes the path of \textit{order} $n$, $M_n$ is a maximum matching on $n$ vertices, and $K_{m,\infty}$ means that the larger partite set has a countable set of vertices. Sometimes we will consider the graph on $n$ vertices which is the disjoint union of a maximum number of copies of $G$, plus isolated vertices in the remainder, and we will call this the \textit{maximal union} of $G$. 

Let $\mathcal F$ be a family of graphs and let $G$ be a graph. We denote by $\mathrm{ex}^G(n,\mathcal F)$ the maximum number of edges in an $\mathcal F$-free graph which contains a copy of $G$, and we write $\mathrm{EX}^G(n,\mathcal F)$ for the corresponding extremal graphs. We write $\mathrm{ex}_c(n,\mathcal F)$ for the \textit{connected extremal number} of $\mathcal F$, the maximum number of edges in a connected $\mathcal F$-free graph, and $\mathrm{EX}_c(n,\mathcal F)$ for the set of corresponding extremal graphs.

\subsection{Main results}

Our first goal is to relate the extremal graphs and spectral extremal graphs. In \cite{ZYY} it was proven that, if $\mathrm{EX}(n,F)=K_k+\overline{K_{n-k}}$ or $\mathrm{EX}(n,F)=K_k+(K_2\cup\overline{K_{n-k-2}})$, then $\mathrm{SPEX}(n,F)=\mathrm{EX}(n,F)$ if $n$ is large enough. The following theorem generalizes this result (see the discussion following Proposition \ref{counterexample}.) Its statement is technical, so an informal summary of its meaning is in order. For the spectral Tur\'an function, not every edge contributes the same amount to the objective function. Edges whose vertices have more eigenweight contribute more, and since large degree vertices tend to have more eigenweight, vertices of large degree are incentivized. We find that if $\mathrm{ex}(n,\mathcal F)=O(n)$, then the spectral extremal graphs want to have as many dominating (or almost-dominating) vertices as possible. Thus, if $G\in\mathrm{SPEX}(n,\mathcal F)$, we will have $G\supseteq K_{k,n-k}$, where $k$ is the largest integer $j$ such that $K_{j,n-j}$ is $\mathcal F$-free. If $K_{k,n-k}$ is $\mathcal F$-saturated, then it follows that $\mathrm{SPEX}(n,\mathcal F)=\{K_{k,n-k}\}$, which is case (a) of Theorem \ref{Thm big ex spex}. If $K_{k,n-k}$ is not $\mathcal F$-saturated, then one may ask how many edges can be added while remaining $\mathcal F$-free. Under some conditions (corresponding to cases (b)-(f)), the answer to this question is enough to determine $\mathrm{SPEX}(n,\mathcal F)$.

\begin{theorem} \label{Thm big ex spex}
Let $\mathcal F$ be a family of graphs. Suppose that $\mathrm{ex}(n,\mathcal F)=O(n)$, $K_{k+1,\infty}$ is not $\mathcal F$-free, and for $n$ large enough $\mathrm{EX}^{K_{k,n-k}}(n,\mathcal F)\ni H$, where one of the following holds:
\begin{itemize}
    \item[(a)] $H=K_{k,n-k}$
    \item[(b)] $H=K_k+\overline{K_{n-k}}$
    \item[(c)] $H=K_k+(K_2\cup\overline{K_{n-k-2}})$
    \item[(d)] $H=K_k+X$ where $e(X)\le Qn+O(1)$ for some $Q\in[0,3/4)$ and $\mathcal F$ is finite
    \item[(e)] $H=K_k+X$ where $e(X)\le Qn+O(1)$ for some $Q\in[3/4,\infty)$ and $\mathcal F$ is finite
    \item[(f)] (e) holds, all but a bounded number of vertices of $X$ have constant degree $d$, $K_k+(\infty\cdot K_{1,d+1})$ is not $\mathcal F$-free. 
\end{itemize}
If (a), (b), or (c) holds then for $n$ large enough, $\mathrm{SPEX}(n,\mathcal F)=H$.

If (d) or (f) holds then for $n$ large enough, $\mathrm{SPEX}(n,\mathcal F)\subseteq\mathrm{EX}^{K_k+\overline{K_{n-k}}}(n,\mathcal F)$. 

If (e) holds then for $n$ large enough all graphs in $\mathrm{SPEX}(n,\mathcal F)$ are of the form $K_k+Y$, where $e(Y)\ge e(X)-O(n^{1/2})$ and $\Delta(Y)$ is bounded in terms of $\mathcal F$.
\end{theorem}

We make two remarks about Theorem \ref{Thm big ex spex}. First, we show in Section \ref{Section EX SPEX} any graph satisfying (d) must have $e(X)=Qn+O(1)$ for some $Q\in\{0,1/2,2/3\}$. Secondly, Theorem \ref{Thm big ex spex} (d) is sharp in that it fails if $Q=3/4$. The reason is that there are two non-isomorphic trees of order $4$.

\begin{proposition} \label{counterexample} There exists a finite family of graphs $\mathcal F$ such that $\mathrm{EX}(n,\mathcal F)=K_2+X$, where $e(X)=3n/4+O(1)$, and for infinitely many $n$, $\mathrm{SPEX}(n,\mathcal F)\cap\mathrm{EX}(n,\mathcal F)=\emptyset$.
\end{proposition}

We now turn to applications. We would like to make use of the existing literature on the extremal functions $\mathrm{ex}(n,\mathcal F)$ and $\mathrm{ex}_c(n,\mathcal F)$. Thus, note that if $\mathrm{EX}(n,\mathcal F)\ni H\supseteq K_{k,n-k}$ or $\mathrm{EX}_c(n,\mathcal F)\ni H\supseteq K_{k,n-k}$ then $\mathrm{EX}^{K_{k,n-k}}(n,\mathcal F)\ni H$, and moreover in cases (a)-(d) or else if $Q<1$, we have that $K_{k+1,\infty}$ is not $\mathcal F$-free. Furthermore it is not too difficult to show that if $\mathrm{ex}_c(n,\mathcal F)=O(n)$ then $\mathrm{ex}(n,\mathcal F)=O(n)$. This implies that in such cases, if we replace the assumption ``$\mathrm{EX}^{K_{k,n-k}}(n,\mathcal F)\ni H$ and $\mathrm{ex}(n,\mathcal F)=O(n)$'' in Theorem \ref{Thm big ex spex} with either ``$\mathrm{EX}(n,\mathcal F)\ni H$" or ``$\mathrm{EX}_c(n,\mathcal F)\ni H$,'' then we obtain a special case of the result, which generalizes Theorem 1.7 of \cite{ZYY}.
\begin{corollary} \label{ex spex corollary}
    Suppose that $\mathcal F$ is a family of graphs such that, for $n$ large enough, $\mathrm{EX}(n,\mathcal F)\ni H$ (or $\mathrm{EX}_c(n,\mathcal F)\ni H$), where $H$ is one of the graphs in (a)-(d) or (f) of Theorem \ref{Thm big ex spex}, and moreover that $\mathcal F$ is finite in cases (d) and (f) and $Q<1$ in case (f). Then for $n$ large enough, $\mathrm{SPEX}(n,\mathcal F)\subseteq\mathrm{EX}(n,\mathcal F)$ (or $\mathrm{SPEX}(n,\mathcal F)\subseteq\mathrm{EX}_c(n,\mathcal F)$).
\end{corollary}
We list some applications of Theorem \ref{Thm big ex spex} and Corollary \ref{ex spex corollary}. All of our results only apply when $n$ is large enough, and this condition is assumed everywhere in the list below. The literature for specific families $\mathcal F$ contains many results with a specific lower bound on the `large enough' $n$, and thus our results do not always imply the quantitative part of the results that we mention.

\begin{itemize}
    \item \textit{Paths.} Balister, Gy\H{o}ri, Lehel, and Schelp \cite{BGLS} showed that $\mathrm{EX_c}(n,P_{2\ell+2})=K_\ell+\overline{K_{n-\ell}}$ and $\mathrm{EX_c}(n,P_{2\ell+3})=K_\ell+(K_2\cup\overline{K_{n-\ell-2}})$ for any $\ell\ge 1$. Thus, Corollary \ref{ex spex corollary} and Theorem \ref{Thm big ex spex} (b) and (c) give $\mathrm{SPEX}(n,P_{2\ell+2})=K_\ell+\overline{K_{n-\ell}}$ and $\mathrm{SPEX}(n,P_{2\ell+3})=K_\ell+(K_2\cup\overline{K_{n-\ell-2}})$. This result was shown by Nikiforov \cite{N}.
    \item \textit{Matchings.} \sloppy Erd\H{o}s and Gallai \cite{EG} showed that $\mathrm{EX}(n,M_{2k+2})=K_k+\overline{K_{n-k}}$. Thus, $\mathrm{SPEX}(n,M_{2k+2})=K_k+\overline{K_{n-k}}$. This result was shown by Feng, Yu, and Zhang \cite{FYZ} (and also proved in a different way by Csikv\'ari \cite{csikvari}).
    \item \textit{Copies of $P_3$.} Bushaw and Kettle \cite{BK}  showed that $\mathrm{EX}(n,k\cdot P_3)=K_{k-1}+M_{n-k+1}$. Thus, Theorem \ref{Thm big ex spex} (d) gives $\mathrm{SPEX}(n,k\cdot P_3)=K_{k-1}+M_{n-k+1}$ This was shown by Chen, Liu, and Zhang \cite{CLZ}.
    \item \textit{Linear forests.} Lidicky, Liu, and Palmer \cite{LLP} generalized the result of \cite{EG}: if $F=\bigcup_{i=1}^j P_{v_i}$ for $v_i\ge 2$, $j\ge 2$, and at least one $v_i\ne 3$, then $\mathrm{EX}(n,F)=K_k+\overline{K_{n-k}}$ if some $v_i$ is even and $\mathrm{EX}(n,F)=K_k+(K_2\cup\overline{K_{n-k-2}})$ if all $v_i$ are odd, where $k=\left(\sum_{i=1}^j\lfloor v_i/2\rfloor\right)-1$. Thus, $\mathrm{SPEX}(n,F)=K_k+\overline{K_{n-k}}$ if some $v_i$ is even and $\mathrm{SPEX}(n,F)=K_k+(K_2\cup\overline{K_{n-k-2}})$ if all $v_i$ are odd.  This result was shown by Chen, Liu, and Zhang \cite{CLZ}.
    \item \textit{Star forests.} Lidicky, Liu, and Palmer \cite{LLP} investigated the extremal number of $F=\bigcup_{i=1}^kK_{1,d_i}$, where $d_1\ge\cdots\ge d_k$ and $k\ge 2$. In this case it is possible that $\mathrm{EX}(n,\mathcal F)\cap \mathrm{SPEX}(n,\mathcal F)=\emptyset.$ Note that $\mathrm{EX}^{K_{k-1,n-k+1}}(n,\mathcal F)=K_{k-1}+X$, where all but at most one vertex of $X$ has degree $d_k-1$ and any other vertex has degree $d_k-2$; $K_{k,\infty}\supseteq F$; and $K_{k-1}+(\infty\cdot K_{1,d_k})$ is not $\mathcal F$-free. Thus Theorem \ref{Thm big ex spex} (f) gives that $\mathrm{SPEX}(n,\mathcal F)\subseteq\mathrm{EX}^{K_{k-1}+\overline{K_{n-k+1}}}(n,\mathcal F)$. Then one can show using equitable partitions that if $n-k+1$ is even, we actually have $\mathrm{SPEX}(n,\mathcal F)=\{K_{k-1}+X:X\text{ is }(d_k-1)\text{-regular}\}$. This was shown by Chen, Liu, and Zhang \cite{CLZ2}.
    \item \textit{Trees.} At the same time that we were finishing this paper Fang, Lin, Shu, and Zhang \cite{FLSZ} proved several theorems studying trees. These include the results listed below about ``certain small trees". Our result below about ``almost all trees" is complementary to the results in \cite{FLSZ}; their results determine for which trees $T$ $\mathrm{SPEX}(n, T) = K_{k,n-k}$ but the characterization is phrased in terms of coverings and the extremal number of a decomposition-like family.
    \begin{itemize}
        \item \textit{Certain small trees.} Caro, Patk\'os and Tuza \cite{CPT} determined $\mathrm{EX}_c(n,T)$ for many small trees $T$. Let $S_{a_1,\ldots,a_j}$ denote the \textit{spider} obtained by identifying paths of orders $a_1+1,\ldots,a_j+1$ at one endpoint, and let $D_{a,b}$ denote the \textit{double star} on $a+b+2$ vertices whose two non-leaf vertices have degrees $a+1$ and $b+1$. Let $D_{2,2}^*$ be the graph obtained by attaching a leaf to a leaf of $D_{2,2}$. Among their results there are some to which we can apply Theorem \ref{Thm big ex spex} which are not already mentioned above: (i) $\mathrm{EX}_c(n,S_{2,2,1})=K_2+\overline{K_{n-2}}$; (ii) $\mathrm{EX}_c(n,S_{3,1,1})=K_1+M_{n-1}$; (iii) $\mathrm{EX}_c(n,D_{2,2})=K_{2,n-2}$; (iv) $\mathrm{EX}_c(n,D_{2,2}^*)=K_2+\overline{K_{n-2}}$; (v) $\mathrm{EX}_c(n,S_{3,2,1})=K_2+\overline{K_{n-2}}$; (vi) $\mathrm{EX}_c(n,D_{2,3})=K_{2,n-2}$.
    \item \textit{Almost all trees.} Let $T$ be a tree with proper coloring $V=A\sqcup B$, where $k+1=|A|\le |B|$, such that $B$ contains a vertex of degree at least 3, exactly one of whose neighbors is not a leaf. (In the proof of Proposition \ref{counting trees} we will show that `almost all trees' indeed have this property.) It is well-known that $\mathrm{ex}(n,T)=O(n)$. Observe that $K_{k,\infty}$ is $T$-free. To show it is $T$-saturated, it remains to show that adding any edge to $K_{k,\infty}$ creates a copy of $T$. Let $L,R$ be the partite sets with $|L|=k$. Since $A$ has a leaf, $e(R)=0$, or else we could embed $T$ in $G$ by placing this leaf incident to an edge in $R$ and placing the rest of $A$ in $L$. To show $e(L)=0$, suppose $L$ has an edge. Then we can embed $T$ in $G$ by placing at least 2 leaves of $A$ in $R$ and the rest of $A$ in $L$, a contradiction. Thus, Theorem \ref{Thm big ex spex} (a) gives $\mathrm{SPEX}(n,T)=K_{k,n-k}$.
      \end{itemize}
    \item \textit{Spectral Erd\H os-S\'os theorem.} Note that $G$ contains all trees on $2k+2$ vertices if and only if $G$ is not $\mathcal F$-free, where $\mathcal F$ is the set of all graphs which contain all trees on $2k+2$ vertices. Note that $\mathrm{EX}^{K_{k,n-k}}(n,\mathcal F)=K_k+\overline{K_{n-k}}$ (see e.g. Lemma 4.6 of \cite{cioabua2022erdossos}). Thus, we have $\mathrm{SPEX}(n,\mathcal F)=K_k+\overline{K_{n-k}}$. Similarly, if $\mathcal F'$ is the set of all graphs which contain all trees on $2k+3$ vertices, then $\mathrm{SPEX}(n,\mathcal F')=K_k+(K_2\cup\overline{K_{n-k-2}})$. These results were shown by Cioab\u a and the second and third authors \cite{cioabua2022erdossos}.
    \item \textit{Long cycles.} Let $\mathcal F=\{C_\ell,C_{\ell+1},\ldots\}$ be the set of cycles of length at least $\ell\ge 5$. Note $\mathrm{EX}^{K_{k,n-k}}(n,\mathcal F)=K_k+\overline{K_{n-k}}$ if $\ell$ is odd and $\mathrm{EX}^{K_{k,n-k}}(n,\mathcal F)=K_k+(K_2\cup\overline{K_{n-k-2}})$ if $\ell$ is even, where $k=\lfloor(\ell-1)/2\rfloor$; this can be viewed as a special case of a result of Kopylov \cite{Kopylov1976}. Erd\H{o}s and Gallai \cite{EG} showed that $\mathrm{ex}(n,\mathcal F)\le\frac{1}{2}(k-1)(n-1)$. Thus, $\mathrm{SPEX}(n,\mathcal F)=K_k+\overline{K_{n-k}}$ if $\ell$ is odd and $\mathrm{SPEX}(n,\mathcal F)=K_k+(K_2\cup\overline{K_{n-k-2}})$ if $\ell$ is even. This result was shown by Gao and Hou \cite{GH}. In fact, the result also follows from the spectral even cycle theorem \cite{cioabua2022evencycle}, but we have included it to demonstrate the versatility of Theorem \ref{Thm big ex spex}.
    \item \textit{Arithmetic progression of cycles.} Let $\mathcal F$ be the set of cycles of length $\ell$ modulo $r$, $\ell<r$. Bollob\'as \cite{B} showed that $\mathrm{ex}(n,\mathcal F)<r^{-1}((r+1)^r-1)n$. Thus, if $\ell$ is even and at least 5, then $\mathrm{SPEX}(n,\mathcal F)=K_{\ell/2-1}+(K_2\cup\overline{K_{n-\ell/2-1}})$ and if $\ell,r$ are both odd then $\mathrm{SPEX}(n,\mathcal F)=K_{(r+\ell)/2-1,n-(r+\ell)/2+1}$. The first result follows from the spectral even cycle theorem, while the second one appears to be new.
    \item \textit{Interval of even cycles.} Note that $G$ contains $k$ consecutive even cycle lengths if and only if $G$ is not $\mathcal F$-free, where $\mathcal F$ is the set of all graphs which contain $k$ consecutive even cycle lengths. Verstra\"ete \cite{V} proved that $\mathrm{ex}(n,\mathcal F)<3kn$. Thus, $\mathrm{SPEX}(n,\mathcal F)=K_k+(K_2\cup\overline{K_{n-k-2}})$. This result appears to be new.
    \item \textit{Disjoint cycles.} Let $\mathcal F$ be the set of all disjoint unions of $k$ cycles (of possibly different lengths), where $k\ge 2$. Erd\H{o}s and P\'osa \cite{EP} showed that $\mathrm{EX}(n,\mathcal F)=K_{2k-1}+\overline{K_{n-2k+1}}$. Thus, $\mathrm{SPEX}(n,\mathcal F)=K_{2k-1}+\overline{K_{n-2k+1}}.$ This result was shown by Liu and Zhai \cite{LZ}.
    \item \textit{Disjoint long cycles.} Let $\mathcal F$ be the set of all disjoint unions of $k$ cycles, each of length at least $\ell\ge 5$. Harvey and Wood \cite{HW} showed that $\mathrm{ex}(n,\mathcal F)<\frac{2}{3}k\ell n$. Thus, if $\ell$ is even then $\mathrm{SPEX}(n,\mathcal F)=K_{k\ell/2-1}+(K_2\cup\overline{K_{n-k\ell/2-1}})$ and if $\ell$ is odd then $\mathrm{SPEX}(n,\mathcal F)=K_{k(\ell+1)/2-1}+\overline{K_{n-k(\ell+1)/2+1}}$. The first result follows from the spectral extremal result for $t\cdot C_\ell$ \cite{FZL}, while the second result appears to be new.
    \item \textit{Disjoint equicardinal cycles.} Let $\mathcal F$ be the set of all disjoint unions of two cycles of the same length. H\"aggkvist \cite{H} showed the $\mathrm{ex}(n,\mathcal F)<6n$. Thus, $\mathrm{SPEX}(n,\mathcal F)=K_3+\overline{K_{n-3}}$. This result appears to be new.
    \item \textit{Minors.} Note that $G$ is $F$-minor free if and only if $G$ is $\mathcal F$-free, where $\mathcal F$ is the set of all graphs which have an $F$-minor. Suppose that $K_k+\overline{K_{n-k}}$ is $\mathcal F$-saturated; then $\mathrm{EX}^{K_{k,n-k}}(n,\mathcal F)=K_k+\overline{K_{n-k}}$, because $\overline{K_k}+(K_2\cup\overline{K_{\infty}})$ has a $K_k+(K_2\cup\overline{K_{\infty}})$-minor. Mader \cite{M2} showed that $\mathrm{ex}(n,\mathcal F)<(k-1)2^{{k-1\choose 2}-1}n$. Thus $\mathrm{SPEX}(n,\mathcal F)=K_k+\overline{K_{n-k}}$. We list below some special cases.
    \begin{itemize}
        \item The $K_k$-minor free spectral extremal graph is $K_{k-2}+\overline{K_{n-k+2}}$. This was shown by the third author \cite{T}.
        \item If $F$ is obtained by deleting from $K_k$ the edges of disjoint paths, not all of order 3, then the $F$-minor free spectral extremal graph is $K_{k-3}+\overline{K_{n-k+3}}$. This was shown recently by Chen, Liu, and Zhang \cite{CLZ5}.
        \item If $F_k$ is the friendship graph with $k$ triangles then the $F_k$-minor free spectral extremal graph is $K_k+\overline{K_{n-k}}$. This was shown by He, Li, and Feng \cite{HLF}.
    \end{itemize}
    If $\mathrm{SPEX}(n,F)=K_k+\overline{K_{n-k}}$ then $\mathrm{SPEX}(n,\mathcal F)=K_k+\overline{K_{n-k}}$ holds automatically, without making use of Theorem \ref{Thm big ex spex}. The same applies to $\mathrm{SPEX}_\alpha(n,F)$.
    \item \textit{Topological subdivisions.} The extremal result of \cite{M2} is actually stronger than as stated in the application above; it applies even when $\mathcal F$ is the family of all topological subdivisions of $K_k$. A similar argument to the minor case then proves the following: if $\mathcal F$ is the family of all topological subdivisions of $F$, and $K_k+\overline{K_{n-k}}$ is $\mathcal F$-saturated then $\mathrm{SPEX}(n,\mathcal F)=K_k+\overline{K_{n-k}}$.
    \item \textit{Chorded cycles.} Let $\mathcal F$ be the family of all chorded cycles. P\'osa (see e.g. in \cite{G}) showed that $\mathrm{ex}(n,\mathcal F)<2n-3$. Thus $\mathrm{SPEX}(n,\mathcal F)=K_{2,n-2}$. This was shown recently by Zheng, Huang, and Wang \cite{ZHW} and gives an answer to a question posed by Gould (Question 3 in \cite{G}).
    \item \textit{Two disjoint chorded cycles.} Let $\mathcal F_1$ be the family of all graphs made of two disjoint cycles, one of which has a chord, and let $\mathcal F_2$ be the family of all graphs made of two disjoint chorded cycles. Bialostocki, Finkel, and Gy{\'a}rf{\'a}s \cite{BFG} showed that $\mathrm{EX}(n,\mathcal F_1)\ni K_{4,n-4}$ and $\mathrm{EX}(n,\mathcal F_2)\ni K_{5,n-5}$. Thus $\mathrm{SPEX}(n,\mathcal F_1)=K_{4,n-4}$ and $\mathrm{SPEX}(n,\mathcal F_2)=K_{5,n-5}$. This result appears to be new.
    \item \textit{Disjoint chorded cycles.} Let $\mathcal F$ be the family of all graphs made of $k$ disjoint chorded cycles. The result of Gao and Wang \cite{GW} implies that $\mathrm{ex}(n,\mathcal F)=O(n)$ via a well-known bound on max cut. Thus, $\mathrm{SPEX}(n,\mathcal F)=K_{3k-1,n-3k+1}$. This result appears to be new.
    \item \textit{Multiply chorded cycles.} Let $\mathcal F$ be the family of all cycles with $k(k-2)+1$ chords, $k\ge 2$. Gould, Horn, and Magnant \cite{GPM} showed that $\mathrm{ex}(n,\mathcal F)=O(n)$. Thus, $\mathrm{SPEX}(n,\mathcal F)=K_{k,n-k}$. This result appears to be new.
    \item \textit{Cycles with $k$ incident chords.} Let $\mathcal F$ be the family of cycles with $k$ chords all incident to the same vertex. Jiang \cite{J} showed that $\mathrm{EX}(n,\mathcal F)\ni K_{k+1,n-k-1}$. Thus, $\mathrm{SPEX}(n,\mathcal F)=K_{k+1,n-k-1}$. This result appears to be new.
\end{itemize}

\begin{proposition} \label{counting trees} The proportion of all labelled $m$-vertex trees $T$ such that for some $k$, $\mathrm{SPEX}(n,T)=K_{k,n-k}$ for $n$ large enough, approaches $1$ as $m\to\infty$.
\end{proposition} 

Our second goal is to relate the spectral extremal graphs and alpha spectral extremal graphs.

\begin{theorem}\label{thm spex to alpha} Let $\mathcal F$ be a family of graphs containing some bipartite graph $F$ for which $F-v$ is a forest, or such that $\mathrm{ex}(n,\mathcal F)=O(n)$. Suppose that for $n$ large enough, $\mathrm{SPEX}(n,\mathcal F)\ni H$, where either:
\begin{itemize}
    \item[(a)] $H=K_k+\overline{K_{n-k}}$
    \item[(b)] $H=K_k+(K_2\cup\overline{K_{n-k-2}})$.
\end{itemize}
Then for any $\alpha\in(0,1)$, for any $n$ large enough, $\mathrm{SPEX}_\alpha(n,\mathcal F)=H$.
\end{theorem}


We give a list of new and existing results implied by Theorem \ref{thm spex to alpha}; if no reference is given for the required spectral result that means that it falls under the previous list of applications.
\begin{itemize}
    \item \textit{Paths.} We have $\mathrm{SPEX}_\alpha(n,P_{2\ell+2})=K_\ell+\overline{K_{n-\ell}}$ and $\mathrm{SPEX}_\alpha(n,P_{2\ell+3})=K_\ell+(K_2\cup\overline{K_{n-\ell-2}})$. This was shown by Chen, Liu, and Zhang \cite{CLZ4}.
    \item \textit{Matchings.} We have $\mathrm{SPEX}_\alpha(n,M_{2k+2})=K_k+\overline{K_{n-k}}$. This was shown by Yuan and Shao \cite{YS}.
    \item \textit{Linear forests.} Generalizing the previous two results: if $F=\bigcup_{i=1}^jP_{v_i}$ where $v_i\ge 2$, $j\ge 2$, and at least one $v_i\ne 3$, then $\mathrm{SPEX}_\alpha(n,\mathcal F)=K_k+\overline{K_{n-k}}$ if some $v_i$ is even and $\mathrm{SPEX}_\alpha(n,\mathcal F)=K_k+(K_2\cup\overline{K_{n-k-2}})$ if all $v_i$ are odd, where $k=\left(\sum_{i=1}^j\lfloor v_i/2\rfloor\right)-1$. This was shown by Chen, Liu, and Zhang \cite{CLZ4}. 
    \item 
    \textit{Certain small trees.} We have (i) $\mathrm{SPEX}_\alpha(n,S_{2,2,1})=K_k+\overline{K_{n-2}}$; (ii) $\mathrm{SPEX}_\alpha(n,D_{2,2}^*)=K_2+\overline{K_{n-2}}$; (iii) $\mathrm{SPEX}_\alpha(n,S_{3,2,1})=K_2+\overline{K_{n-2}}$. These results appear to be new.
    \item \textit{Spectral Erd\H{o}s-S\'os theorem.} Let $\mathcal F$ be the set of all graphs which contain all trees on $2k+2$ vertices, and let $\mathcal F'$ be the set of all graphs which contain all trees on $2k+3$ vertices. Then $\mathrm{SPEX}_\alpha(n,\mathcal F)=K_k+\overline{K_{n-k}}$ and $\mathrm{SPEX}_\alpha(n,\mathcal F')=K_k+(K_2\cup\overline{K_{n-k-2}})$. These results were shown by Chen, Li, Li, Yu, and Zhang \cite{CLLYZ}.
    \item \textit{Even cycles and consecutive cycles.} Cioab\u a and the second and third authors \cite{cioabua2022evencycle} showed that $\mathrm{SPEX}(n,C_{2\ell+2})=K_\ell+(K_2\cup\overline{K_{n-\ell-2}})$ and $\mathrm{SPEX}(n,\{C_{2\ell+1},C_{2\ell+2}\})=K_\ell+\overline{K_{n-\ell}}$, for $k\ge 2$. Thus, $\mathrm{SPEX}_\alpha(n,C_{2\ell+2})=K_k+(K_2\cup\overline{K_{n-\ell-2}})$ and $\mathrm{SPEX}_\alpha(n,\{C_{2\ell+1},C_{2\ell+2}\})=K_\ell+\overline{K_{n-\ell}}$. This was shown by Li and Yu \cite{LY}.
    \item \textit{Intersecting even cycles.} Let $C_{\ell_1,\ldots,\ell_t}$ be obtained by intersecting the cycles $C_{\ell_1}\ldots,C_{\ell_t}$ at a unique vertex. The second author \cite{D} showed that if $\ell_1,\ldots,\ell_t\ge 2$ and some $\ell_t>2$, then $\mathrm{SPEX}(n,C_{2\ell_1,\ldots,2\ell_t})=K_k+(K_2\cup\overline{K_{n-k-2}})$, where $k=\sum_{i=1}^t(\ell_i-1)$. Thus, $\mathrm{SPEX}_\alpha(n,C_{2\ell_1,\ldots,2\ell_t})=K_k+(K_2\cup\overline{K_{n-k-2}})$. This result appears to be new.
    \item \textit{Long cycles.} If $\mathcal F=\{C_\ell,C_{\ell+1},\ldots\}$ then $\mathrm{SPEX}_\alpha(n,\mathcal F)=K_k+\overline{K_{n-k}}$ if $\ell$ is odd and $\mathrm{SPEX}_\alpha(n,\mathcal F)=K_k+(K_2\cup\overline{K_{n-k-2}})$ if $\ell$ is even, where $k=\lfloor(\ell-1)/2\rfloor$. This result is implied by the result forbidding even cycles or consecutive cycles.
    \item \textit{Arithmetic progression of cycles.} Let $\mathcal F$ be the set of cycles of length $\ell$ modulo $r$, $\ell<r$. If $\ell$ is even and at least 5, then $\mathrm{SPEX}_\alpha(n,\mathcal F)=K_{\ell/2-1}+(K_2\cup\overline{K_{n-\ell/2-1}})$. This result is implied by the result forbidding even cycles.
    \item \textit{Interval of even cycles.} Let $\mathcal F$ be the family of all graphs which contain $k$ consecutive even cycle lengths. Then $\mathrm{SPEX}_\alpha(n,\mathcal F)=K_k+(K_2\cup\overline{K_{n-k-2}})$. This result appears to be new.
    \item \textit{Disjoint cycles.} If $\mathcal F$ is the set of all disjoint unions of $k$ cycles (of possibly different lengths) then $\mathrm{SPEX}_\alpha(n,\mathcal F)=K_{2k-1}+\overline{K_{n-2k+1}}$. This was shown by Li, Yu, and Zhang \cite{LYZ}.
    \item \textit{Disjoint long cycles.} Let $\mathcal F$ be the set of all disjoint unions of $k$ cycles, each of length at least $\ell\ge 5$. If $\ell$ is even then $\mathrm{SPEX}_\alpha(n,\mathcal F)=K_{k\ell/2-1}+\overline{K_{n-k\ell/2+1}}$ and if $\ell$ is odd then $\mathrm{SPEX}_\alpha(n,\mathcal F)=K_{k(\ell+1)/2-1}+(K_2\cup\overline{K_{n-k(\ell+1)/2+1}})$. This result seems to be new.
    \item \textit{Disjoint equicardinal cycles.} If $\mathcal F$ is the family of all disjoint unions of two cycles of the same length, then $\mathrm{SPEX}_\alpha(n,\mathcal F)=K_3+\overline{K_{n-3}}$ for $n$ large enough. This result seems to be new.
    \item \textit{Minors.} If $\mathcal F$ is the family of all graphs which have an $F$-minor and $K_k+\overline{K_{n-k}}$ is $\mathcal F$-saturated, then $\mathrm{SPEX}_\alpha(n,\mathcal F)=K_k+\overline{K_{n-k}}$. 
    \begin{itemize}
        \item If $F=K_k$ then $\mathrm{SPEX}_\alpha(n,\mathcal F)=K_{k-2}+\overline{K_{n-k+2}}$. This was shown by Chen, Liu, and Zhang \cite{CLZ3}.
        \item If $F$ is obtained by deleting from $K_k$ the edges of disjoint paths, not all of length 3, then $\mathrm{SPEX}_\alpha(n,\mathcal F)=K_{k-3}+\overline{K_{n-k+3}}$. This result appears to be new.
        \item If $F$ is the friendship graph with $k$ triangles $\mathrm{SPEX}_\alpha(n,\mathcal F)=K_k+\overline{K_{n-k}}$. This was shown by Wang and Zhang \cite{WZ}.
    \end{itemize}
    \item \textit{Topological subdivisions.} If $\mathcal F$ is the family of all topological subdivisions of $F$ and $K_k+\overline{K_{n-k}}$ is $\mathcal F$-saturated, then $\mathrm{SPEX}_\alpha(n,\mathcal F)=K_k+\overline{K_{n-k}}$.
\end{itemize}

We observe that there are fewer options for the assumed extremal graph in Theorem \ref{thm spex to alpha} than in Theorem \ref{Thm big ex spex}. In the case that $H=K_{k,n-k}$ there is a simple counterexample that explains this: let $F$ be any graph such that $\mathrm{SPEX}(n,F)=K_{k,n-k}$ for $n$ large enough, where $k>1$. Then $K_{1,n}$ is also $F$-free, and if $\alpha$ is sufficiently close to 1 then $\lambda_\alpha(K_{1,n})\ge\alpha(n-1)>\lambda_\alpha(K_{k,n-k})$ (see e.g. \cite{nikiforov2017merging}). Thus, one could only hope to extend Theorem \ref{thm spex to alpha} to this case if $\alpha$ is restricted to a certain range. When $H$ is of the form $K_k+X$, then the situation may be more subtle and we do not attempt to address every case.

\section{Spectral extremal results} \label{Section EX SPEX}

 In this section we prove results concerning $\mathrm{SPEX}(n,\mathcal F)$. In subsections \ref{subsection common}, \ref{subsection common 2}, \ref{subsection eigenweights}, and \ref{subsection proof} we prove Theorem \ref{Thm big ex spex}; in subsection \ref{subsection counterexample} we prove Proposition \ref{counterexample}; and in subsection \ref{subsection forbidden} we prove Proposition \ref{counting trees}. Unless stated otherwise, all claims are asserted only for $n$ large enough.

\subsection{Structure of edge extremal graphs} \label{subsection common}

\begin{proposition} \label{matching argument 1} Under the assumptions of Theorem \ref{Thm big ex spex} (d) $X$ can be obtained from a maximal union of $P_1$, of $P_2$, or of $P_3$ by adding or deleting a bounded number of edges.
\end{proposition}
\begin{proof}
Note that $K_{k+1,\infty}\supseteq F$ for some $F\in\mathcal F$. Let $\delta$ be the smallest degree in the partite set of $F$ of size $k+1$. Note that $\Delta(X)\le\delta - 1$, because $F\subseteq K_k+(K_{1,\delta}\cup\overline{K_{n-k-\delta-1}})$.

Consider the components of $X$. We claim the number of components of order at least 4 is bounded; otherwise, the union of these components is a subgraph $Y\subseteq X$ with $|V(Y)|\to\infty$ and $e(Y)\ge 3|V(Y)|/4$, thus contradicting the fact that $\mathrm{ex}^{K_{k,n-k}}(n,\mathcal F)\le kn+Qn+O(1)$ for $n$ large enough. Moreover, the order of a component of $X$ is bounded; otherwise, by taking the largest component $Y$, we find a subgraph $Y\subseteq X$ with $|V(Y)|\to\infty$ and $e(Y)\ge|V(Y)|-1$, again contradicting our bound on $\mathrm{ex}^{K_{k,n-k}}(n,\mathcal F)$. Finally, the number of $K_3$ components of $X$ is bounded; otherwise we find an unbounded $Y\subseteq X$ with $e(Y)=|V(Y)|$, again a contradiction.

Now let $\nu=\max\{i:\text{there exist an unbounded number of components of $X$ of order }i\}\in\{1,2,3\}$. If $\nu=1$ then $e(X)$ is bounded, and we are done.

Suppose $\nu=2$. Then $X$ is the union of an unbounded number copies of $P_2$, a bounded number of graphs of bounded order, and isolated vertices. Suppose for contradiction there are 2 isolated vertices in $X$; then adding an edge $e_0$ between them creates a copy $F_0$ of a graph from $\mathcal F$. Since there were already an unbounded number of disjoint edges in $X$, one such edge $e_1$ is disjoint from $F_0$, and the neighborhoods of its endpoints contain the neighborhoods of the endpoints of $e_0$, so we can replace $e_0$ by $e_1$ to find $F_1\simeq F_0$ which already existed in $K_k+X$, a contradiction. Thus there is at most 1 isolated vertex. Hence, $X$ is obtained from $M_{n-k}$ by deleting and adding a bounded number of edges.

Suppose $\nu=3$. Then $X$ is the union of an unbounded number of copies of $P_3$, a bounded number of graphs of bounded order, and graphs of order less than 3. We will show that the components smaller than $P_3$ span at most 5 vertices; for otherwise, we can find either 3 $P_1$ components, a $P_1$ component and a $P_2$ component, or three $P_2$ components. In any case we can add and delete edges for a net increase in $e(X)$ without creating any components on more than 3 vertices, and such that the copy of some $F_0\in\mathcal F$ which must have been created uses a vertex from a newly created $P_3$ component; but similarly to the previous case, we find another $P_3$ component (or two disjoint $P_3$ components if needed) which was not used in $F_0$, hence there was some $F_1\simeq F_0$ which already existed, a contradiction. This shows that $X$ is obtained from a maximal union of $P_3$ by deleting and adding a bounded number of edges.
\end{proof}

\subsection{Common features of the spectral extremal graphs} \label{subsection common 2}

The purpose of this subsection is to prove the following proposition. 

\begin{proposition} \label{common result} Assume the setup of any of the cases (a)-(f) of Theorem \ref{Thm big ex spex}. If $G\in\mathrm{SPEX}(n,\mathcal F)$ then $G\supseteq K_{k,n-k}$.
\end{proposition}

Let $H$ be the edge extremal graph; when we need to specify the number of vertices, we will write $H_n$ instead of $H$. Set $\lambda=\lambda_1(G)$, and let $x$ be a Perron eigenvector of $G$, scaled so that its greatest entry is $x_z=1$. Let $F\in\mathcal F$ be a graph such that $F\subseteq K_{k+1,\infty}$, and let $m=|V(F)|$,  $m'=\sup\{|V(F')|:F'\in\mathcal F\}$. Let $C$ be chosen such that $\mathrm{ex}(n,\mathcal F)\le Cn$. We choose positive constants $\eta,\varepsilon,\sigma,\beta$ to satisfy the following inequalities:
\begin{multicols}{2}
\begin{itemize}
    \item[(i)] $\beta<\frac{k}{2C}$
    \item[(ii)] $10C\sigma<\varepsilon^2$
    \item[(iii)] $\varepsilon<\eta\le C$
    \item[(iv)] $\sigma<\frac{\varepsilon^2}{30k}$
    \item[(v)] $(8k^3+2)\varepsilon<\eta$
    \item[(vi)] $\varepsilon<(2/5-1/4)/k^3$
    \item[(vii)] $(k+\varepsilon)\eta+\varepsilon^2/2<1/2$
    \item[(viii)] $(2C+m)\eta<1-\frac{1}{4k^3}.$
\end{itemize}
\end{multicols}
To see that such a choice exists, we choose $\eta$ first with the additional constraint that $\eta<k/4$; this allows that some $\varepsilon>0$ will satisfy (vii). Then we choose $\varepsilon,\sigma,$ and $\beta$ in that order. With $V=V(G)$, let $L'=\{v\in V:x_v\ge\eta\}$, $L=\{v\in V:x_v \ge\sigma\}$, $M=\{v\in V:x_v\ge\beta\sigma\}$, and $S=V-L$. For $v\in V$, let $L_i(v)=L\cap N_i(v)$, $M_i(v)=M\cap N_i(v)$, $S_i(v)=S\cap N_i(v)$. 

The proof of Proposition \ref{common result} is technical; a high-level outline is as follows.
\begin{itemize}
    \item[1.] Show that $|L|$ and $|M|$ are not too large, i.e. there are not too many vertices with large eigenweight.
    \item[2.] Show that $|L|$ is in fact bounded.
    \item[3.] Relate the eigenweight of vertices to their degrees.
    \item[4.] Show that eigenweights of vertices in $L'$ are close to 1.
    \item[5.] Show that $|L'|=k$.
    \item[6.] Show that if $u\not\in L'$ then $u\sim v$ for all $v\in L'$.
\end{itemize}

\begin{lemma} \label{initial lambda estimate} We have
$$\sqrt{k(n-k)}\le\lambda\le\sqrt{2Cn}.$$
\end{lemma}
\begin{proof}
Since $K_{k,n-k}$ is $\mathcal F$-free we have
$$\lambda\ge\lambda_1(K_{k,n-k})=\sqrt{k(n-k)}.$$
For the upper bound,
$$\lambda^2=\lambda^2x_{z}=\sum_{u\sim z}\sum_{w\sim u}x_w\le\sum_{u\sim z}d(u)\le\sum_{u\in V}d(u)=2e(G)\le 2Cn.$$
\end{proof}

\begin{lemma}\label{initial L estimate} We have
$$|L|\le\frac{2Cn}{\sigma\sqrt{k(n-k)}},\ \ |M|\le\frac{2Cn}{\sigma\beta\sqrt{k(n-k)}}.$$
\end{lemma}
\begin{proof}
For any $v\in V$ we have $\lambda x_v=\sum_{u\sim v}x_u\le d(v)$. Hence
$$\sqrt{k(n-k)}|L|\sigma\le\lambda|L|\sigma\le\sum_{v\in L}\lambda x_v\le\sum_{v\in V}\lambda x_v\le\sum_{v\in V}d(v)=2e(G)\le 2Cn.$$
Solving for $|L|$ proves the first claim. A similar argument applies to $|M|$.
\end{proof}

\begin{lemma}\label{initial degree estimate} For any $v\in L$,
$$d(v)\ge\frac{\sigma k}{4(1+3C)}n.$$
\end{lemma}
\begin{proof}
Let $v\in L$, $N_1=N_1(v)$, and $M_2=M_2(v)$. We have
$$\sigma k(n-k)\le\lambda^2x_v=\sum_{u\sim v}\sum_{w\sim u}x_w\le d(v)+2e(N_1)+e(N_1,M_2)+\underbrace{\sum_{u\sim v}\sum_{\substack{w\sim u\\w\in N_2-M_2}}x_w}_B$$
from the eigenvector equation, where in the first three terms in the right hand side we used the estimate $x_w\le 1$. We estimate $e(N_1)$, $e(N_1,M_2)$, and $B$ as follows. First, $N_1$ induces a subgraph of order $d(v)$ which does not contain any element of $\mathcal F$, so $e(N_1)\le Cd(v)$. Similarly we have $e(N_1,M_2)\le C\left(\frac{2Cn}{\sigma\beta\sqrt{k(n-k)}}+d(v)\right)$. For $B$, we use $x_w\le\sigma\beta$ and so we have $B\le\sigma\beta e(G)\le\sigma\beta Cn$. Rearranging the inequality, we get
$$\left(\sigma k-\sigma\beta C\right)n-\sigma k^2-\frac{2C^2n}{\sigma\beta\sqrt{k(n-k)}}\le (1+3C)d(v).$$
By inequality (i), taking $n$ large enough gives
$$d(v)\ge\frac{\sigma k}{4(1+3C)}n.$$
\end{proof}

\begin{lemma} There is a constant $C_1=C_1(\mathcal F)$ such that $|L|\le C_1$.
\end{lemma}
\begin{proof}
We use the same argument from Lemma \ref{initial L estimate} with the updated information from Lemma \ref{initial degree estimate}. We have
$$|L|\frac{\sigma k}{4(1+3C)}n\le\sum_{v\in L}d(v)\le 2e(G)\le 2Cn $$
so
$$|L|\le\frac{8C(1+3C)}{\sigma k}=:C_1.$$
\end{proof}

\begin{lemma} \label{vertex eigenweight equation} Let $v\in V,S_1=S_1(v)$, and $L_i=L_i(v)$. Then
$$k(n-k)x_v\le d(v)x_v+e(S_1,L_1\cup L_2)+\frac{\varepsilon^2 n}{2}.$$
\end{lemma}
\begin{proof} Let $c=x_v$. We have
$$\begin{aligned}
    k(n-k)c&\le \lambda^2c=\sum_{u\sim v}\sum_{w\sim u}x_w=d(v)c+\sum_{u\sim v}\sum_{\substack{w\sim u\\w\ne v}}x_w\\
    &\le d(v)c+\sum_{u\in S_1(v)}\sum_{\substack{w\sim u\\w\in L_1\cup L_2}}x_w+2e(S_1)\sigma+2e(L)+e(L_1,S_1)\sigma+e(N_1,S_2)\sigma.
\end{aligned}$$
We apply $e(G)\le Cn$ and $|L|\le C_1$ to get
$$2e(S_1)\sigma+2e(L)+e(L_1,S_1)\sigma+e(N_1,S_2)\sigma\le 2Cn\sigma+2C_1^2+Cn\sigma+Cn\sigma\le 5Cn\sigma.$$
So
$$\begin{aligned}
    k(n-k)c&\le d(v)c+\sum_{u\in S_1}\sum_{w\sim u,w\in L_1\cup L_2}x_w+5Cn\sigma\le d(v)c+e(S_1,L_1\cup L_2)+\frac{\varepsilon^2n}{2}
\end{aligned}$$
where the last inequality holds by (ii).
\end{proof}

\begin{lemma} If $v\in L'$ and $c=x_v$, then $d(v)\ge (c-\varepsilon)n$.
\end{lemma}
\begin{proof}
Let $S_1=S_1(v)$ and $L_i=L_i(v)$. Assume for a contradiction that $d(v)<cn-\varepsilon n$. Thus Lemma \ref{vertex eigenweight equation} gives
$$k(n-k)c\le(cn-\varepsilon n)c+e(S_1,L_1\cup L_2)+\frac{\varepsilon^2 n}{2}.$$
Hence
$$e(S_1,L_1\cup L_2)\ge (k-c+\varepsilon)nc-ck^2-\frac{\varepsilon^2n}{2}=(k-c)nc+\varepsilon nc-ck^2-\frac{\varepsilon^2n}{2}\ge (k-1)nc+\frac{\varepsilon^2n}{2},$$
using that $v\in L'$ implies that $c\ge\eta>\varepsilon$ and $c\le 1$.

Now we claim there are at least $\sqrt n$ vertices in $S_1$ with at least $k$ neighbors in $L_1\cup L_2$. For if not, then $e(S_1,L_1\cup L_2)<(k-1)|S_1|+|L|\sqrt n<(k-1)(c-\varepsilon)n+C_1\sqrt n$ because $|S_1|\le d(v)$ and $|L|\le C_1$. Thus $(k-1)cn-(k-1)\varepsilon n+C_1\sqrt n>e(S_1,L_1\cup L_2)\ge(k-1)cn+\frac{\varepsilon^2n}{2}$ which is a contradiction for $n$ large enough. Thus, let $D$ be a set of $\sqrt n$ vertices in $S_1$ each with at least $k$ neighbors in $L$. Since there are only ${|L|\choose k}$ options for these $k$ neighbors, there is some set of $k$ vertices in $L_1\cup L_2$ with at least $\sqrt n/{|L|\choose k}$ common neighbors in $D$. Since $|L|\le C_1$ is bounded, we have $\sqrt{n}/{|L|\choose k}\to\infty$ as $n\to\infty$, and so by adding the vertex $v$ we find that $G\supseteq K_{k+1,\ell}$ for arbitrarily large $\ell$. However $K_{k + 1, \ell}\supseteq F\in\mathcal F$ for $\ell$ large enough, which is a contradiction.
\end{proof}

\begin{lemma} \label{z vertex lemma}
Let $z$ be the vertex with $x_z=1$, $S_1=S_1(z)$, and $L_i=L_i(z)$. We have $(1-\varepsilon)kn\le e(S_1,\{z\}\cup L_1\cup L_2)\le(k+\varepsilon) n$.
\end{lemma}
\begin{proof}
For the lower bound, we use Lemma \ref{vertex eigenweight equation}:
$$k(n-k)(1)\le d(z)+e(S_1,L_1\cup L_2)+\frac{\varepsilon^2n}{2}\le e(S_1,\{z\}\cup L_1\cup L_2)+C_1+\frac{\varepsilon^2n}{2}\le e(S_1,\{z\}\cup L_1\cup L_2)+k\varepsilon n-k^2$$
where the second inequality follows from $e(S_1,\{z\})\ge d(z)-C_1$ and the last inequality follows from $\varepsilon<2k$. For the upper bound, suppose for a contradiction that $e(S_1,\{z\}\cup L_1\cup L_2)>(k+\varepsilon) n$. Let $\gamma=\varepsilon/C_1$. We claim there are at least $\gamma n$ vertices inside $S_1$ with at least $k$ neighbors in $L_1\cup L_2$. If not then
$$e(S_1,L_1\cup L_2)<(k-1)|S_1|+|L|\gamma n\le (k-1)n+\varepsilon n=(k+\varepsilon-1)n$$ because $|S_1|\le n$ and $|L|\le C_1$. Since $d(z)\le n$, we then have $e(S_1,\{z\}\cup L_1\cup L_2)\le (k+\varepsilon) n$, contradicting the assumption. Hence there is a subset $D\subseteq S_1$ with at least $\gamma n$ vertices such that every vertex in $D$ has at least $k$ neighbors in $L_1\cup L_2$. Since there are at most ${|L|\choose k}$ options for these $k$ neighbors, there exists some set of $k$ vertices in $L_1\cup L_2$ with at least $\gamma n/{|L|\choose k}$ common neighbors in $S_1$. Since $|L|\le C_1$, we have $\gamma n/{|L|\choose k}\to\infty$ as $n\to\infty$, so adding in the vertex $z$ means that $G\supseteq K_{k+1,\ell}$ for arbitrarily large $\ell$, hence contains $F$, which is a contradiction.
\end{proof}

\begin{lemma} \label{large-weight-vertex-degree} For all $v\in L'$, $d(v)\ge\left(1-\frac{2}{5k^3}\right)n$ and $x_v\ge 1-\frac{1}{4k^3}$. 
\end{lemma}
\begin{proof}
We first show $x_v\ge 1-\frac{1}{4k^3}$. Assume for a contradiction there is some $v\in L'$ with $x_v<1-\frac{1}{4k^3}$. Then the eigenvector equation for $z$ gives
$$\begin{aligned}
k(n-k)&\le\lambda^2<e(S_1(z),\{z\}\cup L_1(z)\cup L_2(z)-\{v\})+|N_1(z)\cap N_1(v)|x_v+\frac{\varepsilon^2n}{2}\\
&<(k+\varepsilon)n-|S_1(z)\cap N_1(v)|+|N_1(z)\cap N_1(v)|\left(1-\frac{1}{4k^3}\right)+\frac{\varepsilon^2n}{2}\\
&=kn+\varepsilon n+|L_1(z)\cap N_1(v)|-|N_1(z)\cap N_1(v)|\frac{1}{4k^3}+\frac{\varepsilon^2n}{2}
\end{aligned}$$
where in the first line, the $\varepsilon^2n/2$ term absorbs both the eigenweights of the vertices in $S$ (using inequality (iv)) and the weights from the edges in $L$ since $|L|\le C_1$. This implies
$$\frac{1}{4k^3} |N_1(z)\cap N_1(v)|<\varepsilon n+\frac{\varepsilon^2n}{2}+|L|+k^2\le 2\varepsilon n.$$
However, since $v\in L'$ we have $x_v\ge\eta$ and so $d(v)\ge(\eta-\varepsilon)n$, so $|N_1(z)\cap N_1(v)|\ge(\eta-2\varepsilon)n.$ This contradicts inequality (v). Now since $x_v\ge 1-\frac{1}{4k^3}$, we have $d(v)\ge(x_v-\varepsilon)n\ge\left(1-\frac{1}{4k^3}-\varepsilon\right)n\ge\left(1-\frac{2}{5k^3}\right)n$ using inequality (vi).
\end{proof}

\begin{lemma} We have $|L'|=k$.
\end{lemma}
\begin{proof}
If $|L'|\ge k+1$ then the common neighborhood of $k+1$ vertices in $L'$ would have at least $n-(k+1)\frac{2}{5k^3}n$ vertices, so $G$ contains $K_{k+1,\ell}$ for arbitrarily large $\ell$, a contradiction. 

On the other hand if $|L'|\le k-1$ then a refinement of Lemma \ref{vertex eigenweight equation} applied to $z$ along with Lemma \ref{z vertex lemma} gives the following contradiction,
$$k(n-k)\le d(z)+e(S_1,L'-\{z\})+e(S_1,(L_1\cup L_2)-L')\eta+\frac{\varepsilon^2 n}{2}\le (k-1)n+(k+\varepsilon)n\eta+\frac{\varepsilon^2n}{2}<(k-1/2)n,$$
where the last inequality follows from (vii).
\end{proof}

Since $|L'|=k$ and every vertex in $L'$ has degree at least $\left(1-\frac{2}{5k^3}\right)n$, the common neighborhood of $L'$ has at least $(1-\frac{2}{5k^2})n$ vertices. Let $R$ be this common neighborhood and $E=V-L'-R$. Note that $|E|\le\frac{2}{5k^2}n$. The following lemma completes the proof of Proposition \ref{common result}.

\begin{lemma} \label{lemma proof of common result}
We have $E=\emptyset$.
\end{lemma}
\begin{proof} 
Suppose otherwise. Note that, for each $v\in E$, we have $d_R(v)\le m$, or else by the definition of $m$ we can find a copy of $F$. Hence,
$$\sum_{v\in E}d_{E\cup R}(v)\le 2e(E)+e(E,R)\le 2C|E|+m|E|.$$
Hence, there exists some $v\in E$ such that $d_{E\cup R}(v)\le 2C+m$. Consider the graph $G'$ obtained by deleting all edges between $v$ and $E\cup R$, and adding any missing edges between $v$ and $L'$. The decrease in $x^TAx$ is at most $(2C+m)x_v\eta$, and the increase is at least $x_v\left(1-\frac{1}{4k^3}\right)$, hence $\lambda(G')>\lambda(G)$ by inequality (viii). We proceed by cases according to whether $\mathcal F$ is finite. This is necessary because, in Case 1 below, we need the fact that $|V(F_0)|$ is bounded.

\textit{Case 1: $\mathcal F$ is finite.} Since $\lambda(G')>\lambda(G)$, $G'\supseteq F_0$ for some $F_0\in\mathcal F$, and $F_0$ must use the vertex $v$. Since $|V(F_0)|\le m'<\infty$, there is some $u\in R-V(F_0)$, so by replacing $v$ with $u$ we find some $F_1\subseteq G$ which is isomorphic to $F_0$, a contradiction.

\textit{Case 2: $\mathcal F$ is infinite.} Let $E_1=E-v$, and note that $G'[E_1]$ is $\mathcal F$-free. Furthermore, any vertex in $E_1$ has neighbors only in $L'$, $E_1$, and $R$. Then similarly to the argument above we find $v_2\in E_1$ such that $d_{V-L'}(v_2)\le 2C+m$. Delete all edges between $v_2$ and $V-L'$ and add all edges between $v_2$ and $L'$ to obtain $G''$, and note that $x^TAx$ again increases, so that $\lambda(G'')>\lambda(G)$. Set $E_2=E_1-v_2$, and continue in this way until we obtain $E_i=\emptyset$. Since the resulting graph $G^{(i)}$ satisfies $\lambda(G^{(i)})>\lambda(G)$, $G^{(i)}$ is not $\mathcal F$-free. We now consider whether we are in case (a), (b), or (c) in the statement of Theorem \ref{Thm big ex spex}.

\textit{Subcase (a).} Since $G^{(i)}\supseteq K_{k,n-k}$ and is not $\mathcal F$-free, $G^{(i)}$ must have an edge either in $L'$ or $R$; but since $|R|\to\infty$, this implies that $G^{(i)}[L'\cup R]$ is not $\mathcal F$-free. Since no edges in $L'\cup R$ were modified to get from $G$ to $G^{(i)}$, this implies $G[L'\cup R]$ is not $\mathcal F$-free, a contradiction.

\textit{Subcase (b).} Note that $e(G^{(i)}[R])=e(G[R])\le {k\choose 2} - e(G[L'])\le{k\choose 2}$. Thus we can delete all edges in $E(G^{(i)}[R])$ and add the same number of edges in $L'$ to obtain a subgraph of $H$ (which is therefore $\mathcal F$-free). However, it is easy to see that this operation further increases $x^TAx$ so that the spectral radius of the resulting graph exceeds $\lambda(G)$, which is a contradiction.

\textit{Subcase (c).} Noting $e(G^{(i)}[R])\le {k\choose 2} - e(G[L']) + 1\le{k\choose 2}+1$, the same argument, where we now delete all but one edge in $E(G^{(i)}[R])$, gives the required contradiction.

\end{proof}

{We observe that our arguments for the case where $\mathcal F$ is infinite apply when $H=K_k+\overline{K_{n-k}}$ or $H=K_k+(K_2\cup K_{n-k-2})$, but not when $H=K_k+X$ where $e(X)\ge 2$. This is because, after deleting the edges in $R$ and adding the same number of edges in $L'$, there is more than one possible graph spanned by the remaining edges in $R$, hence we do not necessarily end up with a subgraph of $H$.}

\subsection{Eigenweight estimates} \label{subsection eigenweights}

We now know that there is a set $L'$ of $k$ vertices which are adjacent to every vertex in $R=V(G) \setminus L'$. Any other vertex may have at most $m-1$ neighbors in $V(G) \setminus L'$, otherwise $G$ contains a $K_{k+1, m}$ and is not $F$-free. In this subsection we give more refined estimates on the eigenvector entries.

\begin{lemma}\label{lemma: first eigenweight estimate} For $v\in R$ we have $x_v\le\frac{k}{\lambda - m}$.
\end{lemma}
\begin{proof} Let $v=\arg\max\{x_v:v\in R\}$. Then we have
$$\lambda x_v=\sum_{\substack{u\sim v\\u\in L}}x_u+\sum_{\substack{u\sim v\\u\in R}}x_u\le k+mx_v.$$
\end{proof}

\begin{lemma} \label{eigenweight upper bound} For $v\in R$ we have $x_v\le k\left(\frac{1}{\lambda}+\frac{d_R(v)}{\lambda(\lambda-m)}\right).$
\end{lemma}
\begin{proof}
We have
$$\lambda x_v\le k+\sum_{w\in N_R(v)}x_w\le k+d_R(v)\frac{k}{\lambda-m}$$
by Lemma \ref{lemma: first eigenweight estimate}.
\end{proof}

\begin{lemma} \label{eigenweight lower bound 1} For $v\in R$, we have $x_v\ge\frac{k}{\lambda}-O(n^{-1}).$
\end{lemma}
\begin{proof}
First we find a lower bound on eigenweights for vertices in $L'$. Note that
$$\lambda=\lambda x_z\le k+\sum_{v\in R}x_v\Longrightarrow \sum_{v\in R}x_v\ge \lambda-k.$$
Thus, for any $w\in L'$ we have
$$\lambda x_w\ge\sum_{v\in R}x_v\ge\lambda-k\Longrightarrow x_w\ge 1-O(n^{-1/2}).$$
Hence, for $v\in R$ we have
$$\lambda x_v\ge\sum_{w\in L'}x_w\ge k\left(1-O(n^{-1/2})\right)\Longrightarrow x_v\ge\frac{k}{\lambda}-O(n^{-1}).$$
\end{proof}

Combining the estimates of Lemma \ref{eigenweight upper bound} and Lemma \ref{eigenweight lower bound 1}, we see that for $v\in R$, $x_v=\frac{k}{\lambda}+\Theta(n^{-1})$.
\subsection{Proof of Theorem \ref{Thm big ex spex}} \label{subsection proof}
Note that Proposition \ref{common result} already proves Theorem \ref{Thm big ex spex} (a). Thus, let $G$ be the spectral extremal graph for any remaining case (b)-(f). For notational convenience, let $X=\overline{K_{n-k}}$ and $Q=0$ in case (b), and let $X=K_2\cup\overline{K_{n-k-2}}$ and $Q=0$ in case (c). This way, we can write $H=K_k+X$ in all cases. Note that, since we now know that $G\supseteq K_{k,n-k}$, we have $e(G)\le\mathrm{ex}^{K_{k,n-k}}(n,\mathcal F)=e(K_k+X)$.

\begin{lemma} \label{Lemma SPEX L' clique} The set $L'$ induces a clique.
\end{lemma}
\begin{proof}
Suppose not. We transform $G$ into $K_k+X$ by deleting all edges in $R$ and then adding all edges in $X$ and in $G[L']$. Note $e(R)\le e(X) + \binom{k}{2}$, since $G\supseteq K_{k,n-k}$ and $H\in\mathrm{EX}^{K_{k,n-k}}(n,\mathcal F)$. At least one added edge was in $L'$ and thus increased $x^TAx$ by at least $\left(1-\frac{1}{4k^3}\right)^2>\frac{1}{2}$. Thus the net change in $x^TAx$ is at least
$$\frac{1}{2}+e(X)\left(\frac{k}{\lambda}-O(n^{-1})\right)^2-\left(e(X)+{k\choose 2}\right)\left(\frac{k}{\lambda}+O(n^{-1})\right)^2$$
$$
\ge\frac{1}{2}-e(X)\cdot O\left((\lambda n)^{-1}\right)-{k\choose 2}\left(\frac{k}{\lambda}+O(n^{-1})\right)^2>\frac{1}{2}-Q\cdot O(n^{-1/2})-{k\choose 2}\cdot O(n^{-1})>0$$
which is a contradiction.
\end{proof}

Lemma \ref{Lemma SPEX L' clique} now gives that all vertices in $L'$ are dominating and hence all have eigenvector entry equal to $1$. This means that the error term in Lemma \ref{eigenweight lower bound 1} can be removed, and 
\begin{equation}\label{eigenweight small terms refined}
    x_v\ge\frac{k}{\lambda}.
\end{equation}
Armed with this we may now prove the theorem. 

\begin{proof}[Proof Theorem \ref{Thm big ex spex} (e), (f)] We first only assume only (e). We already know that $G=K_k+Y$ for some $Y$. Now, we can delete all edges in $R$ and add the edges in $X$ to transform $G$ into $X$. The change in $x^TAx$ is at least
$$\begin{aligned}
-e(Y)\left(\frac{k}{\lambda}+O(n^{-1})\right)^2+e(X)\left(\frac{k}{\lambda}\right)^2&\ge(e(X)-e(Y))\frac{k^2}{\lambda^2}-e(Y)O(n^{-3/2}))\\
&\ge(e(X)-e(Y))\frac{k^2}{\lambda^2}-O(n)\cdot O(n^{-3/2}).
\end{aligned}$$
The last quantity above cannot be positive, and so $e(X)-e(Y)=O(n^{1/2})$, which proves the first claim. We already know that $\Delta(Y)$ is bounded (by $m$).

Now assume (f). We first improve bounds on $x_v$ for $v\in Y$. Note that all but a bounded number of $v\in R=V(Y)$ have $d_Y(v)\le d$. By the above, $e(X)-e(Y)=O(n^{1/2})$, so the number of vertices in $Y$ with degree less than $d$ is $O(n^{1/2})$. Call a vertex $v\in R$ \textit{bad} if $d_R(v)\ne d$ or $v\sim u$ and $d_R(u)\ne d$ for some $u\in Y$, and \textit{good} otherwise. Since $\Delta(R)$ is bounded, the number of bad vertices in $R$ is $O(n^{1/2})$. Now for any good vertex $v$,
$$\begin{aligned}
    \lambda^2 x_v&=\sum_{u\in L'}\lambda x_u+dk+\sum_{u\in N_R(v)}\sum_{w\in N_R(u)}x_w\\
    &\ge k\lambda+dk+d^2\frac{k}{\lambda}\\
    x_v&\ge \frac{k}{\lambda}+\frac{kd}{\lambda^2}+\frac{kd^2}{\lambda^3}.
\end{aligned}$$
On the other hand,
$$\begin{aligned}
\lambda^2 x_v&=\sum_{u\in L'}\lambda x_u+dk+\sum_{u\in N_R(v)}\sum_{w\in N_R(u)}x_w\\
&\le k\lambda +dk+d^2\frac{k}{\lambda-m}\\
x_v&\le\frac{k}{\lambda}+\frac{kd}{\lambda^2}+\frac{kd^2}{\lambda^2(\lambda-m)}.
\end{aligned}$$
Combining the two estimates one obtains that $x_v=\frac{k}{\lambda}+\frac{kd}{\lambda^2}+\frac{kd^2}{\lambda^3}+O(n^{-2})$. Now we delete all edges in $Y$ and add all edges in $X$. Note that at most $m\cdot O(n^{1/2}) = O(n^{1/2})$ edges of $Y$ are incident to a bad vertex. It follows that for some $N=O(n^{1/2})$ , there is a set of $N$ edges of $X$ and a set of $N$ edges of $Y$ which contain all edges of $X$ or $Y$, respectively, incident to a bad vertex. Then the change in $x^TAx$ is at least
$$\begin{aligned}
    &-(e(Y)-N)\left(\frac{k}{\lambda}+\frac{kd}{\lambda^2} +\frac{kd^2}{\lambda^3}+O(n^{-2})\right)^2+(e(X)-N)\left(\frac{k}{\lambda}+\frac{kd}{\lambda^2}  +\frac{kd^2}{\lambda^3}\right)^2\\
    &-N\left(\frac{k}{\lambda}+O(n^{-1})\right)^2+N\left(\frac{k}{\lambda}\right)^2\\
    &\ge (e(X)-e(Y))\frac{k^2}{\lambda^2}-e(Y)\frac{k}{\lambda}\cdot O(n^{-2}) -N\frac{k}{\lambda}\cdot O(n^{-1})\\
    &\ge(e(X)-e(Y))\cdot O(n^{-1})-Qn\cdot O(n^{-5/2})-O(n^{1/2})\cdot O(n^{-3/2}).
\end{aligned}$$
If $e(X)-e(Y)=\omega(1)$ then the quantity above is positive, a contradiction. Hence $e(X)-e(Y)=O(1)$, which implies that the number of bad vertices is in fact $O(1)$. Therefore we can take $N=O(1)$ and repeat the previous argument to show that when we transform $K_k+Y$ to $K_K+X$ the change in $x^TAx$ is at least
$$(e(X)-e(Y))\cdot O(n^{-1})-Qn\cdot O(n^{-5/2})-O(1)\cdot O(n^{-3/2})$$
which implies $e(X)-e(Y)=0$; hence $G\in\mathrm{EX}^{K_k+\overline{K_{n-k}}}(n,\mathcal F)$.
\end{proof}

\begin{proof}[Proof of Theorem \ref{Thm big ex spex} (b), (c), (d)] From Proposition \ref{matching argument 1}, we have $e(X)=Qn+O(1)$ for some $Q\in\{0,1/2,2/3\}$. The first case is $Q=0$, in other words $e(X)=O(1)$ and consequently $e(R)=O(1)$. If $G\not\in\mathrm{EX}^{K_k+\overline{K_{n-k}}}(n,\mathcal F)$, then delete all edges in $R$ and add all edges in $X$ to make the graph isomorphic to $K_k+X$. Let $s,t$ be the number of edges deleted and added, respectively, which satisfy $s<t=O(1)$. Then the change in $x^TAx$ is at least
$$-s\left(\frac{k}{\lambda}+O(n^{-1})\right)^2+t\left(\frac{k}{\lambda}\right)^2\ge(t-s)\frac{k^2}{\lambda^2}-O\left((\lambda n)^{-1}\right)-O(n^{-2})>0$$
which is a contradiction. Actually, this argument proves the following fact: whenever $K_k+X$ can be obtained from $K_k+Y$ by deleting a bounded number of edges and adding more edges than were deleted, we have $\lambda_1(K_k+X)>\lambda_1(K_k+Y)$.

The second case is $Q=1/2$. We first note that since $G$ is $\mathcal F$-saturated, only a bounded number of vertices $v\in R$ have $d_R(v)\ge 2$, by the same argument as Proposition \ref{matching argument 1}. On the other hand, $e(R)$ must be unbounded, or the same argument as in the case $Q=0$ would show $\lambda_1(G)<\lambda_1(K_k+X)$. Therefore, the argument in Proposition \ref{matching argument 1} shows that $G[R]$ has at most one isolated vertex. Thus $G[R]$ is obtained from a maximal union of $P_2$ by deleting and adding a bounded number of edges. Since the same is true of $K_k+X$, we can delete and add a bounded number of edges to transform $G$ into a graph isomorphic to $K_k+X$. By the fact mentioned above, this implies that $G\in\mathrm{EX}^{K_k+\overline{K_{n-k}}}(n,\mathcal F)$.

The final case $Q=2/3$ is similar. We find that $Y$ is the union of copies of $P_3$ and smaller graphs, plus a bounded number of graphs of bounded order. The number of copies of $P_3$ must be unbounded; otherwise, we have $e(Y)<\frac{2}{3}n$, and we may substitute $t-s=\Omega(n)$ in the argument above, and find
$$-s\left(\frac{k}{\lambda}+O(n^{-1})\right)^2+t\left(\frac{k}{\lambda}\right)^2\ge\Omega(n)\frac{k^2}{\lambda^2}-O(n)\cdot((\lambda n)^{-1})-O(n)\cdot O(n^{-2})>0$$
which gives a contradiction. The argument of Proposition \ref{matching argument 1} then shows that $G$ is obtained by deleting and adding a bounded number of edges to a maximal union of $P_3$. By the fact mentioned above, this implies that $G\in\mathrm{EX}^{K_k+\overline{K_{n-k}}}(n,\mathcal F)$.
\end{proof}

{Note that the conclusions stated in (b) and (c) follow because in these cases $\mathrm{EX}^{K_k+\overline{K_{k,n-k}}}(n,\mathcal F)$ contains a single graph up to isomorphism.}

\subsection{A counterexample} \label{subsection counterexample}

\begin{proof}[Proof of Proposition \ref{counterexample}]
Let $\mathcal F$ consist of the all the following graphs: 

\begin{enumerate}
    \item $3\cdot K_{1,4}$;
    \item $K_{2,6}\cup X_5\cup Y_5$, for all connected graphs $X_5,Y_5$ on 5 vertices;
    \item $K_{2,6}\cup K_{1,3}\cup X_5$, for all connected graphs $X_5$ on 5 vertices;
    \item $K_{2,6}\cup X_9$, for all connected graphs $X_9$ on 9 vertices;
    \item $K_{2,6}\cup X_{8}$, for all connected graphs $X_{8}$ on 8 vertices except for $P_8$;
    \item $X_5\cup Y_5\cup Z_5\cup W_5$, for all connected graphs $X_5,Y_5,Z_5,W_5$ on 5 vertices;
    \item $C_i \cup C_j \cup C_k$ for $3\leq i,j,k \leq 7$.
\end{enumerate}

We have chosen a very large family $\mathcal F$, which makes the extremal problem for $\mathcal F$ easy. We suspect that a similar counterexample could be obtained with an even smaller family $\mathcal F$.

\begin{lemma} \label{extremal graphs for counterexample} If $n\equiv 2\pmod 4$ is large enough, then we have that $\mathrm{EX}(n,\mathcal F)=K_2+\left(P_8\cup\frac{n-10}{4}\cdot P_4\right)$.
\end{lemma}
\begin{proof}
One can check that $H=K_2+\left(P_8\cup\frac{n-10}{4}\cdot P_4\right)$ is $\mathcal F$-free. Now let $H'$ be the extremal graph, so that
$$e(H')\ge 2(n-2)+1+7+\frac{n-10}{4}\cdot 3=\frac{11}{4}n-4+1+7-\frac{15}{2}=\frac{11}{4}n-\frac{7}{2}.$$
By forbidding (1), we may choose vertices $x,y$ such that any other vertex $z$ has at most $11$ neighbors in $V(G)\setminus \{x,y\}$. Consider the components of $G-\{x,y\}$; these are either graphs on 3 or 4 vertices containing a $C_3$ or $C_4$, respectively (and by (7) there are at most 2 of each type), or connected graph of order at least $5$ (and by (6) there are at most 3 of these), or trees on at most $4$ vertices. We claim that the order of any graphs of the second type is at most $5\cdot 11^9$. Let $K$ be such a component with $|V(K)|\ge 5\cdot 11^9$; let $B_r(v)$ denote $\{u\in V(K):d_K(u,v)\le r\}$. Then for any $v$, $|B_8(v)|\le 1+11+\cdots+11^8\le 11^9$ since $\Delta(K)\le 11$, and consequently we can find vertices $v_1,\ldots,v_5$ such that for $i\ne j$, $d(v_i,v_j)\ge 9$. This implies that $B_4(v_i)$, $1\le i\le 5$, induce 5 disjoint connected graphs on at least 5 vertices, which contradicts (6). Putting this all together, we find a partition $V(H')=U\sqcup B$ where $U$ is a union of trees of order at most 4, $|B|\le 2+2\cdot 3+2\cdot 4+3\cdot 5\cdot 11^9\le 11^{11}$, and $x,y\in B$.
Hence, since $\Delta(G[B-x-y])\le 11$, we have
$$\frac{11}{4}n-\frac{7}{2}\le d(x)+d(y)+11^{11}\cdot 11+\frac{n}{4}\cdot 3\Longrightarrow d(x)+d(y)\ge 2n-\frac{7}{2}-11^{12}\ge 2n-11^{13}.$$
So $d(x)\ge n-11^{13}$, $d(y)\ge n-11^{13}$, and $|N(x)\cap N(y)|\ge n-2\cdot 11^{13}$. Note that $x$ and $y$ are contained in disjoint copies of $K_{1,4}$ so by (1), every other vertex of $H'$ has at most 3 neighbors outside $\{x,y\}$. Now we have
$$e(N(x)\cap N(y))\ge\frac{11}{4}n-\frac{7}{2}-2(n-1)-2\cdot 11^{13}\cdot 5=\frac{3}{4}n-O(1)$$
which implies we can find 2 disjoint edges in $N(x)\cap N(y)$. Hence $x$ and $y$ are contained in disjoint copies of $C_3$. Thus by (5) and (7), all components of $G-x-y$, except possibly one (which by (4) can be of order at most $8$), are trees of order at most 4. Such a graph cannot have more edges than $H$. If $e(H')=e(H)$, then $G-x-y=a\cdot P_4\cup b\cdot K_{1,3}\cup K$, where $K$ is a connected graph on 8 vertices (using (4)). By (5), we have $K=P_8$; but then by (3), $b=0$. It follows that $H'\subseteq H$, hence $H'\simeq H$.
\end{proof}

As in the proof above, let $H=K_2+\left(P_8\cup\frac{n-10}{4}\cdot P_4\right)$. Let $G=K_2+\frac{n-2}{4}\cdot K_{1,3}$. One can check that $G$ is also $\mathcal F$-free. We now show that $\lambda_1(G)>\lambda_1(H)$. Note that $G$ and $H$ have equitable partitions with quotient matrices
$$B_G=
\begin{bmatrix}
    1&\frac{1}{4}(n-2)&\frac{3}{4}(n-2)\\
    2&0&3\\
    2&1&0\\
\end{bmatrix},\ 
B_H=
\begin{bmatrix}
    1&\frac{1}{2}(n-10)&\frac{1}{2}(n-10)&2&2&2&2\\
    2&1&1&0&0&0&0\\
    2&1&0&0&0&0&0\\
    2&0&0&1&1&0&0\\
    2&0&0&1&0&1&0\\
    2&0&0&0&1&0&1\\
    2&0&0&0&0&1&0
\end{bmatrix}.$$
The characteristic polynomials are
$$\begin{aligned}
    p_G(x)&=x^3-x^2+(1-2n)x+9-3n\\
    p_H(x)&=x^7-3x^6+(3-2n)x^5+(3+n)x^4+(7n-22)x^3+(1-n)x^2+(10-4n)x+3-n.
\end{aligned}$$
By the Weyl inequalities, $\lambda_1(G),\lambda_1(H)\ge\sqrt{2(n-2)}$ and $\lambda_2(G),\lambda_2(H)<\sqrt{2(n-2)}$. Thus $p_G$ is negative on $[\sqrt{2(n-2)},\lambda_1(G))$ and positive on $(\lambda_1(G),\infty)$, and $p_H$ is negative on $[\sqrt{2(n-2)},\lambda_1(H))$ and positive on $(\lambda_1(H),\infty)$. So, to show $\lambda_1(G)>\lambda_1(H)$ it suffices to find some $x>\sqrt{2(n-2)}$ such that $p_G(x)<0<p_H(x)$. For a constant $C>0$, we compute
$$\begin{aligned}
    p_G(\sqrt{2(n-2)}+5/4+Cn^{-1/2})&=4C\sqrt n-\frac{13}{8\sqrt 2}\sqrt{n-2}+o(\sqrt n)\\
    p_H(\sqrt{2(n-2)}+5/4+Cn^{-1/2})&=16Cn^{5/2}-\frac{5}{2\sqrt 2}n^2\sqrt{n-2}+o(n^{5/2})
\end{aligned}$$
as $n\to\infty$. Choosing $\frac{13}{32\sqrt 2}<C<\frac{5}{32\sqrt 2}$, we find that $p_G(\sqrt{2(n-2)}+5/4+Cn^{-1/2})\to-\infty$ while $p_H(\sqrt{2(n-2)}+5/4+Cn^{-1/2})\to\infty$.
\end{proof}

\subsection{Forbidden trees} \label{subsection forbidden}

\begin{proof}[Proof of Proposition \ref{counting trees}]
We will assume that all labelled trees have vertex set $[n]$ and let $T$ denote a uniformly random labelled tree on $[n]$. For $i\ne j$ in $[n]$, let $A_{ij}$ be the event that $i\sim j$, $i$ has a neighbor different from $j$ which is adjacent precisely to $i$ and 2 leaves, and $j$ has a neighbor different from $i$ which is adjacent precisely to $j$ and 2 leaves; let $X_{ij}$ be the indicator function of $A_{ij}$. Observe that if $A_{ij}$ occurs then $T$ is of the form described in the application `Almost all trees' of Theorem \ref{Thm big ex spex}, for some $k\ge 4$. Therefore, defining $X=\sum_{i<j} X_{ij}$, it suffices to show that $\mathbb P[X=0]\to 0$. Our argument is a variation on \cite{ER} in which it was shown that almost all labelled trees have a vertex with two leaves.

For $i,j,k,\ell\in[n]$ (not necessarily distinct), we define $N_n(ij)$ to be the number of labelled trees on $[n]$ which contain the edge $ij$, and we define $N_n(ij,k\ell)$ to be the number of labelled trees on $[n]$ which contain both edges $ij$ and $k\ell$. From Lemma 6 of \cite{LMS}, we have the following formulas.

\begin{lemma} \label{tree completion} For any $i,j,k,\ell$ distinct we have:
\begin{itemize}
    \item[(a)] $N_n(ij)=2n^{n-3}$
    \item[(b)] $N_n(ij,ik)=3n^{n-4}$
    \item[(c)] $N_n(ij,k\ell)=4n^{n-4}$.
\end{itemize}
\end{lemma}

Next we estimate $\mathbb E[X]$ and $\mathbb E[X^2]$. To specify a tree $T$ such that $A_{ij}$ occurs, we choose the neighbor of $i$ and its pair of leaves, then we choose the neighbor of $j$ and its pair of leaves, and then attach a tree containing the edge $ij$ and the unused vertices in $[n]$. For a constant $c$ we have that $(n-c)^{n-c} = n^{n-c}e^{-c} + \Theta(n^{n-c-2})$. This gives:
$$\begin{aligned}
    \mathbb P[A_{ij}]&=\frac{1}{n^{n-2}}(n-2){n-3\choose 2}(n-5){n-6\choose 2}N_{n-6}(ij)=\frac{1}{2e^6n}+o(n^{-1})\\
    \mathbb E[X]&={n\choose 2}\left(\frac{1}{2e^6n}+o(n^{-1})\right)=\frac{n}{4e^6}+o(n)\\
    \mathbb E[X]^2&=\frac{n^2}{16e^{12}}+o(n^2).
\end{aligned}$$
Next we consider $A_{ij}\cap A_{ik}$. {There are 2 types of trees which realize this event: (i) those in which the special neighbor of $i$ in $A_{ij}$ is the same as the special neighbor of $i$ in $A_{ik}$ (the special neighbor of $j$ and $k$ must be distinct in a tree), and (ii) those in which they are different. So we have
$$\begin{aligned}
  \mathbb P[A_{ij}\cap A_{ik}]&\le\frac{n^{9}3n^{n-13}+n^{12}3n^{n-16}}{n^{n-2}}=O(n^{-2}).
\end{aligned}$$
}
Finally, if $i,j,k,\ell$ are distinct, then we have
$$\begin{aligned}\mathbb P[A_{ij}\cap A_{k\ell}]&\le \frac{n^4{n\choose 2}^44(n-12)^{n-16}}{n^{n-2}}=\frac{1}{4e^{12}n^2}+o(n^{-2}).
\end{aligned}$$
Therefore, we have
$$\begin{aligned}
\mathbb E[X^2]&=\sum_{i<j}\mathbb P[A_{ij}]+\sum_{i,j,k}\mathbb P[A_{ij}\cap A_{ik}]+\sum_{\{i,j\},\{k,\ell\}\text{ disjoint}}\mathbb P[A_{ij}\cap A_{k\ell}]\\
&\le {n\choose 2}\frac{1}{e^6n}+n^3\cdot O(n^{-2})+{n\choose 2}{n-2\choose 2}\left(\frac{1}{4e^{12}n^2}+o(n^{-2})\right)\\
&=\frac{n^2}{16e^{12}}+o(n^2).
\end{aligned}$$
Therefore $\mathrm{Var}(X)=o(n^2)$, and Chebyshev's inequality gives
$$\mathbb P[X=0]\le\mathbb P\left[|X-\mathbb E[X]|\ge\frac{n}{5e^6}\right]\le\frac{25e^{12}\cdot o(n^2)}{n^2}\to 0.$$
\end{proof}

\section{Alpha spectral extremal results} \label{Section SPEX ALPHA SPEX}

Let $\mathcal F$ be the family of graphs in Theorem \ref{thm spex to alpha}. We first consider the case where $\mathcal F$ contains some bipartite graph $F$ with a vertex $v$ such that $F-v$ is a forest. 

\begin{lemma}
\label{Forbidden forest}
    There exists $c=c(F)$ such that, for any graph $G$ with $V(G)=U\sqcup W$ and
    \[2e(U) + e(U,W) > 3c |U| + c |W|,\]
    there exists a copy of $F-v$ in $G$ such that $N_F(v)$ embeds in $U$.
\end{lemma}

\begin{proof} Let $v$ be the vertex such that $F-v$ is a forest, and let $\Delta=\Delta(F-v)$. Let $V(F)=A\sqcup B$ be a proper coloring with $v\in B$. Then $F-v$ has a proper coloring $A\sqcup(B\setminus \{v\})$, and $A$ contains all the neighbors of $v$ in $F$.  Set $a=|A|$ and $b=|B|$.

Let $T$ be a forest obtained from 2 copies of $F-v$ by selecting a vertex from $A$ from each component (observe that every component of $F-v$ must contain a vertex in $A$) and making it adjacent to the corresponding vertex in the other copy of $F-v$. Now if $t=|V(T)|=2(|V(F)|-1)$ then it is known that $\mathrm{ex}(n,T)\le (t-2)n$.

Now consider the graph $G$ with $V(G)=U\sqcup W$, appearing in the statement of this lemma. If $e(U)>(t/2-2)|U|$ then $G[U]\supseteq F-v$ and we are done. If $e(U,W)>(t-2)(|U|+|W|)$ then $G[U,W]\supseteq T$, and by the design of $T$ this means there is a copy of $F-v$ in $G$ in which $A$ embeds in $U$. Hence, to guarantee the conclusion of the lemma, it suffices that $e(U)>(t-2)|U|$ or $e(U,W)>(t-2)(|U|+|W|)$, and we may take e.g. $c=t$.
\end{proof}

If $G$ is any $F$-free graph, then for any $u\in V(G)$, there cannot be a copy of $F-v$ in $G[N_1(u)\cup N_2(u)]$ such that all the neighbors of $N_F(v)$ embed in $N_1(u)$. Applying Lemma \ref{Forbidden forest} gives

\begin{equation}
\label{upper bound on edges in N_1, <N_1, N_2>}
2e(N_1(v)) + e(N_1(v), N_2(v)) \le 3c d(v) + c (n - d(v) - 1) = 2cd(v) + c(n-1) \le 3cn.
\end{equation}

If $\mathcal F$ does not contain a graph $F$ so that $F-v$ is a forest, then we assume that $\mathrm{ex}(n,\mathcal F)=O(n)$. By choosing $c$ large enough in Equation \ref{upper bound on edges in N_1, <N_1, N_2>}, we may assume that the equation holds in either case, for some $c=c(\mathcal F)$, for any $\mathcal F$-free graph.

\begin{lemma}
(\cite{nikiforov2017merging})
\label{bounds for alpha spectral radius of any graph}
    Let $G$ be a graph and $0 \le \alpha \le 1$, then 
    \[\alpha \Delta(G) \le \lambda_{\alpha}(G) \le \alpha\Delta(G) + (1 - \alpha)\lambda(G).\]
\end{lemma}

\begin{lemma} \label{upper bound for lambda(G)} For any $\mathcal F$-free graph $G$ we have $\lambda(G)\le\sqrt{k(n-k)}+k$.
\end{lemma}
\begin{proof}
We have $\lambda(G)\le\mathrm{spex}(n,\mathcal F)\le\lambda(K_k+(K_2\cup\overline{K_{n-k-2}}))$, and by the Weyl inequalities we have
$$\lambda(K_k+(K_2\cup\overline{K_{n-k-2}}))\le\lambda(K_{k,n-k})+k=\sqrt{k(n-k)}+k.$$
\end{proof}

\begin{lemma}
\label{lemma useful eigenvalue-eigenvector equation}
    Let $G$ be a connected graph and let $y$ be a Perron vector of $A_{\alpha}(G)$. Then for all $v \in V(G)$, we have
    $$\begin{aligned}
    \lambda_\alpha(G)y_v&=\alpha d(v)y_v+(1-\alpha)\sum_{u\sim v}y_u\\
    \lambda_{\alpha}^2(G)y_v &= \alpha d(v) \lambda_{\alpha}(G) y_v + \alpha(1-\alpha)\sum_{u \sim v}d(u) y_u + (1-\alpha)^2 \sum_{w \sim v}\sum_{u \sim w}y_u\\
    \lambda_\alpha(G)&=\frac{1}{y^Ty}\left(\alpha\sum_{v\in V(G)}d(v)y_v^2+(1-\alpha)\sum_{uv\in E(G)}y_uy_v\right).
    \end{aligned}$$
\end{lemma}

Now let $G$ be a graph in $\mathrm{SPEX}_\alpha(n,\mathcal F)$ and let $x$ be its Perron vector, scaled so that its largest entry is $x_z=1$. Let $\sigma$ be a positive constant satisfying $3c\sigma<1-\alpha$ and $(3c+1)\sigma < \alpha(1-\alpha)$. With $V=V(G)$, let $L=\{v\in V:x_v>\sigma\}$ and let $S=V-L$.  Also, write $L_i(v)=L\cap N_i(v)$ and $S_i(v)=S\cap N_i(v)$. As in Section \ref{Section EX SPEX}, all claims below are only asserted for $n$ large enough unless stated otherwise.

\begin{lemma}
   \label{bounds on alpha spex} 
   For $n$ sufficiently large, we have that
   \[\rspex_\alpha(n, \mathcal F) \ge \max\left\{\alpha n + \frac{k}{\alpha} - k - 1 -\frac{2k(k+1)}{\alpha^3 n - \alpha^2(k+1+\alpha) + \alpha k}, \alpha n + \frac{k}{\alpha} - k - 1 - \alpha, \alpha(n-1)\right\}.\]
\end{lemma}
\begin{proof}
To obtain the lower bound, observe that we have assumed $K_k+\overline{K_{n-k}}$ is $\mathcal F$-free. Thus, $\raspex(n, \mathcal F) \ge \lambda_{\alpha}(K_k+\overline{K_{n-k}})$ and
\[\lambda_{\alpha}(K_k+\overline{K_{n-k}}) \ge \max\left\{\alpha n + \frac{k}{\alpha} - k - 1 -\frac{2k(k+1)}{\alpha^3 n - \alpha^2(k+1+\alpha) + \alpha k}, \alpha n + \frac{k}{\alpha} - k - 1 - \alpha, \alpha(n-1)\right\},\]
see for example \cite{CLLYZ}.
\end{proof}

\begin{lemma}
\label{degree of vertices in L}
    For all $v \in L$, we have $d(v) \ge \left(1 - \dfrac{1}{(k+1)\sqrt{n}}\right)n$.
\end{lemma}

\begin{proof}
Suppose to the contrary that there is some $v\in L$ with $d(v) < \left(1 - \dfrac{1}{(k+1)\sqrt{n}}\right)n$.
Then, 
$$
    \begin{aligned}
        \lambda_{\alpha}^2x_v &= \alpha d(v) \lambda_{\alpha} x_v + \alpha(1-\alpha)\sum_{u \sim v}d(u) x_u + (1-\alpha)^2 \sum_{w \sim v}\sum_{u \sim w}x_u\\
        &\le \alpha d(v) \lambda_{\alpha} x_v + \alpha(1-\alpha)(d(v) + 2e(N_1(v)) + e(N_1(v), N_2(v))) \\
        &+ (1-\alpha)^2 (d(v) + 2e(N_1(v)) + e(N_1(v), N_2(v)))\\
        &< \alpha d(v) \lambda_{\alpha} x_v + \alpha(1-\alpha)(d(v) +  3cn) + (1-\alpha)^2(d(v) +  3cn)\\
        &< \alpha d(v) \lambda_{\alpha} x_v + (1-\alpha)(3c+1)n.
    \end{aligned}
$$

Then 
$$
    \lambda_{\alpha}(\lambda_{\alpha} - \alpha d(v))x_v < (1 - \alpha)(3c+1)n.
$$
On the other hand, from Lemma \ref{bounds on alpha spex} and $d(v) < \left(1 - \dfrac{1}{(k+1)\sqrt{n}}\right)n$ we have
$$\lambda_{\alpha}(\lambda_{\alpha} - \alpha d(v))x_v > \alpha(n-1)\left[\alpha(n-1) - \alpha\left(1 - \dfrac{1}{(k+1)\sqrt{n}}\right)n\right]\sigma=\Omega(n^{3/2}).$$
which gives a contradiction.
\end{proof}

\begin{lemma}
\label{size of color class of F}
    There is some $F'\in\mathcal F$ such that $F'\subseteq K_{k+1,\infty}$. Consequently, $F' \subseteq K_{k+1,m}$ for some fixed $m$.
\end{lemma}
\begin{proof}
   If not, then $K_{k+1, n-k-1}$ is $\mathcal F$-free. Since $\lambda(K_{k+1, n-k-1}) > \lambda(K_k+(K_2\cup\overline{K_{n-k-2}})$, this contradicts the fact that $\mathrm{spex}(n,\mathcal F)\le\lambda(K_k+(K_2\cup\overline{K_{n-k-2}})$.
\end{proof}

\begin{lemma}
\label{exact size of L}
The size of $L$ is $k$.
\end{lemma}
\begin{proof}
First assume for a contradiction that $|L|\ge k+1$. Then there is a set $L' \subseteq L$ with $L' = k+1$. Therefore, every vertex $v\in L'$ has $d(v) \ge \left(1 - \dfrac{1}{(k+1)\sqrt{n}}\right)n$, and the common neighborhood of the vertices in $L'$ has at least $n - \sqrt{n}$ vertices. This implies that for $n$ large enough, $G\supseteq K_{k+1,n-\sqrt n}$ which contradicts Lemma \ref{size of color class of F}.

To show $|L|\ge k$, assume to the contrary that $|L| \le k-1$. In cases (b) and (c), we have
    $$\begin{aligned}
         \lambda_{\alpha}^2 &= \lambda_{\alpha}^2x_z = \alpha d(z) \lambda_{\alpha} x_z + \alpha(1-\alpha)\sum_{u \sim z}d(u) x_u + (1-\alpha)^2 \sum_{w \sim z}\sum_{u \sim w}x_u\\
        &\le \alpha d(z) \lambda_{\alpha} + \alpha(1-\alpha)\left(\sum_{u \in L_1(z)}d(u) x_u + \sum_{u \in S_1(z)}d(u) x_u\right) \\
        &+ (1-\alpha)^2  \left(\sum_{w\sim z}\sum_{u\in L_1(w)}x_u+\sum_{w\sim z}\sum_{u\in S_1(w)}x_u \right)\\
        &< \alpha d(z) \lambda_{\alpha} + \alpha(1-\alpha)\left((k-2)n + (d(z) + 2e(N_1(z)) + e(N_1(z), N_2(z)))\sigma\right)\\
        &+ (1-\alpha)^2\left((k-1)n + (2e(N_1(z)) + e(N_1(z), N_2(z)))\sigma\right)\\
        &< \alpha d(z) \lambda_{\alpha} + \alpha(1-\alpha)\left((k-2)n + (d(z) + 3cn)\sigma\right)+(1-\alpha)^2\left(3cn\sigma + (k-1)n\right)\\
        &\le \alpha d(z) \lambda_{\alpha} + (1-\alpha)((k-1)n + 3cn\sigma).
    \end{aligned}$$
Therefore,
$$\begin{aligned}
\lambda_{\alpha}(\lambda_{\alpha} - \alpha d(z)) &< (1-\alpha)(k-1+3c\sigma)n.
\end{aligned}$$
But from Lemma \ref{bounds on alpha spex} we have
$$\begin{aligned}
\lambda_\alpha(\lambda_\alpha - \alpha d(z))&\ge\alpha(n-1)\left(\alpha n+\frac{k}{\alpha}-k-1-\frac{4k^2}{\alpha^4 n}-\alpha(n-1)\right)=(1-\alpha)(k-\alpha)n+O(1)
\end{aligned}$$
using $2k(k+1)/(\alpha^3n-\alpha^2(k+1+\alpha)+ak)\le 4k^2/(\alpha^4n)$. Since $3c\sigma<1-\alpha$, these two inequalities give a contradiction.
\end{proof}

It follows from Lemma \ref{degree of vertices in L} and Lemma \ref{exact size of L} that $L = k$ and for all $v \in L,$ $d(v) \ge \left(1 - \dfrac{1}{(k+1)\sqrt{n}}\right)n$. Let $R$ denote the common neighborhood of all vertices in $L$. Then $|R| \ge\left(1 - \dfrac{k}{(k+1)\sqrt{n}}\right)n > n - \sqrt{n}$. Let $E = S \setminus R$, so that $|E| \le \sqrt{n}$. We will show that $E = \emptyset$ and so $G \supseteq K_{k, n-k}$.

\begin{lemma}
\label{Perron entries of L}
For any $v \in L$, we have $x_v \ge 1-\alpha$.
\end{lemma}
\begin{proof}
Assume to the contrary that there is some $v \in L$ with $x_v < 1-\alpha$. Then the second-degree eigenvector equation with respect to $z$ gives
$$\begin{aligned}
    \lambda_\alpha(\lambda_\alpha-\alpha d(z))&\le \alpha(1-\alpha)\sum_{u\sim z}d(u)x_u+(1-\alpha)^2\sum_{w\sim z}\sum_{u\sim w}x_u\\
    &\le \alpha(1-\alpha)\left(\sum_{u\in L_1(z)}d(u)x_u+\sum_{u\in S_1(z)}d(u)x_u\right)\\
    &\ +(1-\alpha)^2\left(\sum_{w\sim z}\sum_{u\in L_1(w)}x_u+\sum_{w\sim z}\sum_{u\in S_1(w)}x_u\right)\\
    &\le\alpha(1-\alpha)\left((k-1)n+(3c+1)\sigma n\right)+(1-\alpha)^2\left((k-1+x_v)n+(3c+1)\sigma n\right)\\
    &\le \left(\alpha(k-1)+(1-\alpha)(k-\alpha)+(3c+1)\sigma\right)(1-\alpha)n.
\end{aligned}$$

On the other hand, from Lemma \ref{bounds on alpha spex} we have
$\lambda_\alpha(\lambda_\alpha-\alpha d(z))\ge (1-\alpha)(k-\alpha)n+O(1)$ which, since $\alpha(k-1)+(1-\alpha)(k-\alpha)+(3c+1)\sigma<(k-\alpha),$ gives a contradiction.
\end{proof}

Note that since $F'\subseteq K_{k+1,m}$, for any $v\not\in L$ we have $d_R(v)<m$.

\begin{lemma} \label{alpha initial S eigenweight upper bound} For any $v\in S$ we have $x_v=O(n^{-1})$.
\end{lemma}
\begin{proof}
Let $v=\arg\max_{u\in S}x_u$. Then we have
$$\begin{aligned}
    \lambda_\alpha x_v&\le \alpha d(v)x_v+(1-\alpha)(k+(m+\sqrt n)x_v)\\
    &\le \alpha(k+m+\sqrt n)x_v+(1-\alpha)(k+(m+\sqrt n)x_v)\\
    &\le k+(m+\sqrt n)x_v
\end{aligned}$$
so
$$(\lambda_\alpha-m-\sqrt n)x_v\le k\Longrightarrow x_v\le\frac{k}{\lambda_\alpha-m-\sqrt n}=O(n^{-1}).$$
\end{proof}

\begin{lemma} \label{alpha E is empty} We have $E=\emptyset$.
\end{lemma}
\begin{proof}
Delete all edges between $E$ and $E\cup R$ and add all missing edges between $E$ and $L$, to obtain a graph $G'$. The deletions decrease $x^TA_\alpha x$ by at most
$$\alpha|E|(2m+\sqrt n)\cdot O(n^{-2})+(1-\alpha)|E|(m+\sqrt n)\cdot O(n^{-2})=o(1),$$

while the additions increase $x^TA_\alpha x$ by at least $\alpha|E|(1-\alpha)^2.$
If $E\ne\emptyset$ then the net change is positive. Hence, $G'\supseteq F_0$ for some $F_0\in\mathcal F$. This implies that $e(R)\ge 1$ in case (a) and $e(R)\ge 2$ in case (b). Either way, it follows that $e(L)<{k\choose 2}$, since $G$ is $\mathcal F$-free. Delete from $G'$ all edges in $R$ and add an edge in $L$, to obtain a subgraph of $H$. The deletions decrease $x^TA_\alpha x$ by at most
$$\alpha|R|\cdot O(n^{1/2})\cdot O(n^{-2})+(1-\alpha)|R|\cdot O(n^{1/2})\cdot O(n^{-2})=O(n^{-1/2})$$
while the addition increases $x^TA_\alpha x$ by at least $(1-\alpha)^3$, so the net change is again positive. Thus $\lambda_\alpha(G)<\lambda_\alpha(G')<\lambda_\alpha(H)$ which is a contradiction.
\end{proof}

\begin{lemma} \label{alpha left clique} The set $L$ induces a clique.
\end{lemma}
\begin{proof}
Suppose not. Then we transform $G$ into $K_k+\overline{K_{n-k}}$ by deleting all edges in $R$ and adding at least one edge in $L$. The deletions decrease $x^TA_\alpha x$ by at most
$$\alpha nm\cdot O(n^{-2})+(1-\alpha)\frac{nm}{2}\cdot O(n^{-2})=O(n^{-1}).$$
while the additions increase $x^TA_\alpha x$ by at least $(1-\alpha)^3$,
which is a contradiction.
\end{proof}

\begin{proof}[Proof of Theorem \ref{thm spex to alpha}] By Lemma \ref{alpha left clique}, we know $G\supseteq K_k+\overline{K_{n-k}}.$  In case (a), $K_k+\overline{K_{n-k}}$ is $\mathcal F$-saturated, hence $G=K_k+\overline{K_{n-k}}$. In case (b), there is at most one edge in $R$, which implies $G=K_k+(K_2\cup\overline{K_{n-k-2}})$.
\end{proof}

\section{Concluding remarks} \label{Section Conclusion}



Our aim in this paper was to study situations where $\mathrm{SPEX}(n,\mathcal F)\subseteq \mathrm{EX}(n,\mathcal F)$ or $\mathrm{SPEX}_\alpha(n,\mathcal F)\subseteq\mathrm{SPEX}(n,\mathcal F)$. In fact, we have given `recipes' which allow one to determine these extremal graphs in more general situations. Many more applications of these results could be stated, and in our list of applications we have restricted ourselves to families for which the edge extremal graphs were studied in previous papers. There are also applications which follow easily from the methods of this paper but are not strictly consequences of our main theorems. We list some such cases below.
\begin{itemize}
    \item For certain forbidden families $\mathcal F$, our methods can prove somewhat stronger results than what we have stated. First, for some infinite families $\mathcal F$ there is an alternative proof of Lemma \ref{lemma proof of common result} even if $H$ is of the form $K_k+X$ where $e(X)\ge 2$. In such cases, we may remove the assumption that $\mathcal F$ is finite from Theorem \ref{Thm big ex spex} (d) and (e). Second, if we have the characterization
    $$K_k+Y\text{ is }\mathcal F\text{-free}\Longleftrightarrow Y\subseteq Z\text{ for some }K_k+Z\in\mathrm{EX}^{K_k+\overline{K_{n-k}}}(n,\mathcal F)$$
    then we may strengthen the conclusion of Theorem \ref{Thm big ex spex} (e) to $\mathrm{SPEX}(n,\mathcal F)\subseteq\mathrm{EX}^{K_k+\overline{K_{n-k}}}(n,\mathcal F)$ and also prove that $\mathrm{SPEX}_\alpha(n,\mathcal F)\subseteq\mathrm{EX}^{K_k+\overline{K_{n-k}}}(n,\mathcal F)$. We list some additional applications which make use of these facts.
    \begin{itemize}
        \item $\mathrm{SPEX}(n,\{C_4,C_5,\ldots\})=K_1+M_{n-1}.$ This also follows from the spectral extremal result for $C_4$ \cite{Nikiforov07}.
        \item $\mathrm{SPEX}_\alpha(n,C_4)=K_1+M_{n-1}$ \cite{TCC}.
        \item $K_{2k-1}+M_{n-2k+1}$ has the largest (alpha) spectral radius without $k$ disjoint even cycles.
        \item $K_3+M_{n-3}$ has the largest (alpha) spectral radius without 2 disjoint theta graphs; see \cite{gao2015extremal} for the definition and required extremal result.
        \item If $F=\bigcup_{i=1}^kK_{1,d_i}$ is a star forest, then any graph $\mathrm{SPEX}_\alpha(n,F)$ is the join of $K_{k-1}$ and an almost $(d_k-1)$-regular graph; in particular, if $n-k+1$ is even $\mathrm{SPEX}_\alpha(n,F)=\mathrm{SPEX}(n,F)$. These results where shown in \cite{CLZ3}.
        \item $K_2+P_{n-2}$ has the largest spectral radius among planar graphs \cite{tait2017three}. Here we may take $\mathcal F$ to be the set of topological subdivisions of $K_{3,3}$ or $K_5$ \cite{kuratowski1930probleme}. We require one additional modification of the proof: in Lemma \ref{lemma proof of common result}, after obtaining the graph $G^{(i)}$, if the two vertices in $L'$ are not adjacent, then we can delete at most one edge in $R$ and add the edge in $L'$ while preserving planarity of the graph induced by $L'\cup R$ and increasing $x^TAx$. One can check that the entire graph must now be planar.
        \item $K_1+P_{n-1}$ has the largest spectral radius among outerplanar graphs \cite{tait2017three}. Here we may take $\mathcal F$ to be the set of topological subdivisions of $K_{2,3}$ or $K_4$. In Lemma \ref{lemma proof of common result}, we use the fact that adding pendant vertices preserves outerplanarity.
    \end{itemize}
    \item In some cases our methods work for extremal graphs of the form $W+\overline{K_{n-|W|}}$, where $W$ is a fixed graph which is neither empty nor complete. For example, if $\mathcal F$ is finite and every graph in $\mathrm{EX}^{K_{|W|,n-|W|}}(n,\mathcal F)$ contains $\mathrm{ex}(|W|,\mathcal F)$ edges in $W$ and no edges in $V-W$, then under the other assumptions of Theorem \ref{Thm big ex spex}, we have that any graphs in $\mathrm{SPEX}(n,\mathcal F)$ are $W+\overline{K_{n-|W|}}$ where $W\in\mathrm{EX}(n,|W|)$. As consequence, we obtain the spectral analogues of Theorems 1.6 and 1.7 of \cite{LK}.
    \item As mentioned in the introduction, when $\mathrm{SPEX}(n,\mathcal F)=K_{k,n-k}$ then under the assumptions of Theorem \ref{thm spex to alpha} we can show that $\mathrm{SPEX}_\alpha(n,\mathcal F)=K_{k,n-k}$ for a certain range of $\alpha$. However, additional arguments would be needed to characterize $\mathrm{SPEX}_\alpha(n,\mathcal F)$ for every $\alpha\in[0,1)$.
\end{itemize}

Additionally, when our methods do not determine the extremal graphs exactly, one can sometimes still obtain general structural results. For example, let $\mathcal F$ be any finite family of graphs such that $\mathrm{ex}(n,\mathcal F)=O(n)$. Let $k=\max\{\ell:K_{\ell,\infty}\text{ is }\mathcal F\text{-free}\} <\infty.$ Then the proof of Proposition \ref{common result} gives that any $G\in\mathrm{SPEX}(n,\mathcal F)$ has $G\supseteq K_{k,n-k}$. Moreover, any vertex in the partite set of $G$ of size $n-k$ has degree at most $k+m$, where $m$ is the smallest minimum degree of a vertex in the partite set of size $k+1$ of any $F\in\mathcal F$ satisfying $F\subseteq K_{k+1,\infty}$.

We end by listing some open questions.
\begin{itemize}
    \item To prove Theorem \ref{Thm big ex spex} in its generality we require that $\mathrm{ex}(n, \mathcal{F}) = O(n)$. However, many of the techniques in the proof were originally discovered in the proof of the spectral even-cycle theorem \cite{cioabua2022evencycle} where the Tur\'an number is much larger. When considering a specific forbidden family, the structure of the forbidden graph(s) can be used and in some cases the condition on the Tur\'an number can be relaxed. When $\alpha=0$, perhaps the assumption that $F-v$ is a forest for some bipartite $F\in\mathcal F$ could be useful. For example, it would be interesting to try to prove a spectral version of the general result of Faudree and Simonovits \cite{FS}. In any case, it seems possible that our assumption $\mathrm{ex}(n,\mathcal F)=O(n)$ could be weakened or modified, and one could obtain a general result which implies, e.g., the spectral even cycle theorem. 
    \item For the case $\alpha>0$, we do not know whether the assumption that $F-v$ is a forest could be weakened. For example, if one may assume only that $F-\{v_1,\ldots,v_k\}$ is a forest then perhaps one could determine the alpha spectral extremal graphs for $k\cdot C_{2\ell}$.
    \item The counterexample we constructed to prove Proposition \ref{counterexample} is quite complicated and artificial. With respect to the two extremal graphs $H=K_2+(P_8\cup\frac{n-10}{4}\cdot P_4)$ and $G=K_2+\frac{n-2}{4}\cdot K_{1,3}$ obtained in the proof of Proposition \ref{counterexample}, it is natural to pose the following question: does there exist a graph $F$ such that $\mathrm{EX}(n,F)=H$ and $\mathrm{SPEX}(n,F)=K$ for infinitely many $n$? Failing a positive answer to this question, the following problem may be interesting: what is the supremum of constants $Q$ such that for any graph $F$, $\mathrm{EX}(n,F)\ni K_k+X$ with $e(X)\le Qn+O(1)$ implies $\mathrm{SPEX}(n,F)\subseteq\mathrm{EX}(n,F)$ for large $n$?
    \item The only step in the proof of Theorem \ref{Thm big ex spex} which prevents us from allowing $\mathcal F$ to be infinite in all cases is Lemma \ref{lemma proof of common result}. We were unable to modify our arguments to work for all infinite graph families or to find a counterexample showing this cannot be done. In some cases (such as the additional applications listed above), the nature of the graph family $\mathcal F$ allows one to prove Lemma \ref{lemma proof of common result} without resorting to the argument we use for cases (a)-(c). Ad-hoc arguments of this kind may work for many other natural infinite families.
    \item In cases (d) and (e) of Theorem \ref{Thm big ex spex}, it is not too difficult to show that if $Q<1$, then there is some $R$ such that $e(X)=Rn+O(1)$. Can this be shown for $Q\ge 1$?
    \item As mentioned in the introduction, the extremal graphs covered by Theorem \ref{thm spex to alpha} are more restricted than those covered by Theorem \ref{Thm big ex spex}. Our methods should give some information about $\mathrm{SPEX}_\alpha(n,\mathcal F)$ when $\mathcal F$ is as in cases (a), (d), (e), or (f) of Theorem \ref{Thm big ex spex}, but we have not attempted to determine the strongest possible statement one can prove using our methods. 
    \item When $G=K_{k,n-k}$, we have seen that the extremal function $\mathrm{ex}^G(n,\mathcal F)$ is relevant to determining the spectral extrema for many families of forbidden graphs. Perhaps this motivates further study of this or other `restricted extremal problems'; in particular, it would be interesting to see if there are cases in which determination of $\mathrm{EX}^{K_{k,n-k}}(n,\mathcal F)$ is nontrivial while also providing an application of Theorem \ref{Thm big ex spex}.
\end{itemize}

\bibliographystyle{plain}
\bibliography{references}

\end{document}